\documentclass[oneside,english]{amsart}
\usepackage[T1]{fontenc}
\usepackage[latin9]{inputenc}
\usepackage{color}
\usepackage{babel}
\usepackage{amstext}
\usepackage{amsthm}
\usepackage{amssymb}
\usepackage[unicode=true,
 bookmarks=true,bookmarksnumbered=true,bookmarksopen=true,bookmarksopenlevel=1,
 breaklinks=true,pdfborder={0 0 0},pdfborderstyle={},backref=page,colorlinks=true]
 {hyperref}
\hypersetup{pdftitle={Waring's Problem For Locally Nilpotent Groups: The Case of Discrete Heisenberg Groups},
 pdfauthor={Ya-Qing Hu},
 pdfsubject={11P05, 11C08, 20M14, 20F18},
 pdfkeywords={Key words and phrases: Waring's problem, polynomial maps, commutative semigroups, locally nilpotent groups, discrete Heisenberg groups.}}

\makeatletter
\numberwithin{equation}{section}
\numberwithin{figure}{section}
\theoremstyle{plain}
\newtheorem*{question*}{\protect\questionname}
\theoremstyle{plain}
\newtheorem{thm}{\protect\theoremname}
\theoremstyle{plain}
\newtheorem{prop}{\protect\propositionname}
\theoremstyle{plain}
\newtheorem{cor}{\protect\corollaryname}
\theoremstyle{plain}
\newtheorem*{thm*}{\protect\theoremname}
\theoremstyle{definition}
 \newtheorem{example}{\protect\examplename}
\theoremstyle{remark}
\newtheorem*{rem*}{\protect\remarkname}
\theoremstyle{plain}
\newtheorem*{prop*}{\protect\propositionname}
\theoremstyle{plain}
\newtheorem{lem}{\protect\lemmaname}
\theoremstyle{definition}
\newtheorem*{example*}{\protect\examplename}

\DeclareMathOperator{\diag}{diag}

\DeclareMathOperator{\Ima}{Im}

\DeclareMathOperator{\Int}{Int}

\DeclareMathOperator{\rank}{rank}

\usepackage{pgfplots}
\pgfplotsset{compat=1.11}
\usepackage{xcolor}
\usepackage{tikz}
\usepackage{tikz-cd}
\usepackage{tikz-3dplot}
\usetikzlibrary{%
  positioning,%
  patterns,%
  shapes,%
  matrix,%
  calc,%
  arrows%
}
\allowdisplaybreaks

\makeatother

\providecommand{\corollaryname}{Corollary}
\providecommand{\examplename}{Example}
\providecommand{\lemmaname}{Lemma}
\providecommand{\propositionname}{Proposition}
\providecommand{\questionname}{Question}
\providecommand{\remarkname}{Remark}
\providecommand{\theoremname}{Theorem}

\begin{document}
\title[Waring's Problem For Locally Nilpotent Groups]{Waring's Problem For Locally Nilpotent Groups: The Case of Discrete
Heisenberg Groups}
\author{Ya-Qing Hu}
\address{Morningside Center of Mathematics\\
Chinese Academy of Sciences\\
No. 55, Zhongguancun East Road\\
Haidian District, Beijing 100190}
\email{\href{mailto:yaqinghu@amss.ac.cn}{yaqinghu@amss.ac.cn}}
\thanks{This work was partially supported by NSF Grant {[}grant number DMS-1401419{]}
and the Postdoctoral International Exchange Program of the China Postdoctoral
Council {[}grant number YJ20210319{]}. }
\begin{abstract}
Kamke \cite{Kamke1921} solved an analog of Waring's problem with
$n$th powers replaced by integer-valued polynomials. Larsen and Nguyen
\cite{LN2019} explored the view of algebraic groups as a natural
setting for Waring's problem. This paper applies the theory of polynomial
maps and polynomial sequences in locally nilpotent groups developed
in a previous work \cite{Hu2020} to solve an analog of Waring's problem
for the general discrete Heisenberg groups $H_{2n+1}(\mathbb{Z})$
for any integer $n\ge1$. 
\end{abstract}

\date{\today}

\maketitle
\tableofcontents{}

\section{Introduction}

\subsection*{Motivation }

The motivation of this work is the following question of Michael Larsen: 
\begin{question*}
Find good notions of ``polynomial sequence'' and ``generalized cone''
so that if $G$ is a finitely generated nilpotent group and $g_{0},g_{1},g_{2},\ldots$
is a polynomial sequence in $G$ such that no coset of any infinite
index subgroup of $G$ contains the whole sequence, then there exists
a positive integer $M$, a generalized cone $C\subset G$, and a subgroup
$H$ of finite index in $G$ such that every element of $C\cap H$
is a product of $M$ elements of the sequence. 
\end{question*}
A previous work \cite{Hu2020} proposed definitions for polynomial
sequences $g:\mathbb{N}_{0}\to G;i\mapsto g_{i}:=g(i)$ and generalized
cones and proved many desirable formal properties of polynomial sequences
when the target group $G$ is locally nilpotent. The present work
will answer the question in the case of the general discrete Heisenberg
groups $H_{2n+1}(\mathbb{Z})$ for any integer $n\ge1$. 

\subsection*{Background}

Let $\mathbb{N}$ (resp. $\mathbb{N}_{0}$) be the set of positive
(resp. non-negative) integers. In 1909, Hilbert \cite{Hilbert1909}
solved the classical Waring's problem by a difficult combinatorial
argument based on algebraic identities and proved that for each positive
integer $n$, there exists a bounded number $N\in\mathbb{N}$ dependent
only on $n$ such that the following map given by the sum of $n$th
powers of non-negative integers is surjective: 
\[
\mathbb{N}_{0}^{N}\to\mathbb{N}_{0};\quad(x_{1},x_{2},\ldots,x_{N})\mapsto\sum_{i=1}^{N}x_{i}^{n}.
\]
This is known as the Hilbert-Waring theorem. 

On the other hand, various variants of Waring's problem have been
investigated. For example, Kamke \cite{Kamke1921} generalized the
Hilbert-Waring theorem with $x^{n}$ replaced by integer-valued polynomials
$f(x)$ of degree $\ge2$. Wright \cite{Wright1934} studied the easier
Waring's problem, which seeks to determine $v(n)$, the minimum $N$
such that $\mathbb{Z}$ is the union of images of all maps of the
form 
\[
\mathbb{N}_{0}^{N}\to\mathbb{Z};\quad(x_{1},x_{2},\ldots,x_{N})\mapsto\sum_{i=1}^{N}\varepsilon_{i}x_{i}^{n},
\]
for some choices of $\varepsilon_{i}=\pm1$. 

Moreover, the analog of Waring's problem for (nonabelian) groups receives
a great deal of attention in the last 30 years. A typical problem
is to prove that every element in the group $G$ can be expressed
as a short product of values of certain word map
\[
w:\underbrace{G\times G\times\cdots\times G}_{d}\to G,
\]
induced by substitution of a nontrivial group word $w$ in the free
group $F_{d}$ of rank $d$ with elements in the group $G$. 

Recently, Larsen and Nguyen \cite{LN2019} explored the idea of algebraic
groups as a natural setting for Waring's problem. The work on the
polynomial-valued, vector-valued and certain matrix-valued variants
of Waring's problem can naturally fit into this framework. They consider
a morphism of varieties (i.e., reduced separated schemes of finite
type) (resp. schemes) from $\mathbb{A}^{1}$ to an algebraic group
$G$ defined over a field $K$ (resp. a group scheme over a number
ring $\mathcal{O}$).

At the field level, they work in the field of characteristic $0$
and call a subvariety $X$ of the algebraic group $G$ \emph{generating},
if there exists $n\in\mathbb{N}$ such that the product map 
\[
X^{\times n}:=X\times\cdots\times X\to G
\]
is surjective, or equivalently, every generic point of $G$ lies in
the image of this product map, and call a finite collection of morphisms
$f_{i}:X_{i}\to G$ generating if the finite union of Zariski closures
$\overline{f_{i}(X_{i})}$ is generating. 

They are interested in the generating collections of morphisms $f_{i}:\mathbb{A}^{1}\to G$,
and for certain technical reasons, they restrict their attention to
connected unipotent algebraic groups over a nonreal field $K$. (They
call a field $K$ nonreal if $K$ is of characteristic $0$ but not
formally real, i.e., $-1$ is a sum of squares in $K$.) They prove
that for any unipotent algebraic group $G$ over a nonreal field $K$
and a generating set $\{f_{1},\ldots,f_{n}\}$ of $K$-morphisms $\mathbb{A}^{1}\to G$,
there exists some positive integer $M$ such that $\left(f_{1}(K)\cup\cdots\cup f_{n}(K)\right)^{M}=G(K)$;
see \cite[Thm 2.2]{LN2019}. 

At the integral level, they work with the ring $\mathcal{O}$ of integers
of a totally imaginary number field $K$, and a closed $\mathcal{O}$-subscheme
$\mathcal{G}$ of the group scheme $\mathcal{U}_{k}$ of unitary $k\times k$
matrices, and call a set $\{f_{1},\ldots,f_{n}\}$ of $\mathcal{O}$-morphism
$\mathbb{A}^{1}\to\mathcal{G}$ generating if it is generating as
a $K$-morphism. They prove that for any generating set $\{f_{1},\ldots,f_{n}\}$
of $\mathcal{O}$-morphisms $\mathbb{A}^{1}\to\mathcal{G}$, there
exists a positive integer $M$ such that $\left(f_{1}(\mathcal{O})\cup\cdots\cup f_{n}(\mathcal{O})\right)^{M}$
is a subgroup of finite index in $\mathcal{G}(\mathcal{O})$; see
\cite[Thm 3.1]{LN2019}. 

This framework is convenient to work with but also has its own drawbacks.
For example, in the original situation of Waring's problem, namely
for the ring $\mathbb{Z}$, the additive group $\mathbb{G}_{a}$,
and the morphism $f:\mathbb{A}^{1}\to\mathbb{G}_{a}$ given by $f(x)=x^{n}$
for $n\ge2$ even, their results fall short of the Hilbert-Waring
theorem. The difficulty is the ordering of $\mathbb{Z}$, as negative
integers cannot be expressed as sums of positive elements and thus
the assumption of $K$ being totally imaginary plays a vital role
in their work. But this issue can be avoided in the situation of the
easier Waring's problem. Hence, they have to choose between either
working over a totally imaginary number ring or doing the easier Waring
problem on a general number ring. Moreover, a natural question has
been raised in \cite{LN2019}: whether, for unipotent groups over
general number rings $\mathcal{O}$, one can characterize the set
which ought to be expressible as a bounded product of images of the
morphism $f:\mathbb{A}^{1}\to\mathcal{U}_{k}$ over the ring $\mathcal{O}$.

\subsection*{Strategy and main result}

To solve an analog of Waring's problem for nilpotent groups, the original
plan was to do induction over central series of the group and use
Kamke's result as the base case. \cite[Prop 2.4]{LN2019} suggests
a strategy for the induction step, but the ordering on $\mathbb{Z}$
and the lack of inverses make their strategy break down for this work.
In \cite{Hu2020} we propose an alternative strategy called the \emph{iterated
symmetrization}, which enables us to find a symmetric polynomial map
$\mathbb{N}_{0}^{L}\to\mathcal{U}_{n}(\mathbb{Z})$ in $L$ variables,
where the symmetric group $S^{L}$ acts on such polynomial maps by
permuting variables. In particular, we can make the induction method
work and solve an analog of Waring's problem for the general discrete
Heisenberg groups $H_{2n+1}(\mathbb{Z})$. The main result of this
work is the following 
\begin{thm}
\label{thm:WP for Heisenberg group} Let $G$ be the discrete Heisenberg
group $H_{2n+1}(\mathbb{Z})$ for some positive integer $n$ and $g:\mathbb{N}_{0}\to G$
be a polynomial sequence. If the induced polynomial sequence $g\mod N:\mathbb{N}_{0}\to G\twoheadrightarrow G/N$
is non-constant for any normal subgroup $N$ of infinite index in
$G$, then there exists a positive integer $M$, a finite index subgroup
$H$ of $G$, and a proper polynomial set $V$ of $H$ such that every
element in $V$ can be written as a product of at most $M$ elements
in the sequence $g_{0},g_{1},g_{2},\ldots$.
\end{thm}

\subsection*{Organization of the paper}

In Section \ref{sec:2 Integer-Valued Polynomials}, we discuss some
basic properties of integer-valued polynomials, especially on the
basis for integer-valued polynomials $\mathbb{N}_{0}\to\mathbb{Z}$,
state two equivalent forms of Kamke's key theorem, briefly summarize
Kamke's generalization of Waring's problem to integer-valued polynomials,
discuss Diophantine Frobenius problem and the greatest common divisor
of polynomial values, answer the question in Proposition \ref{prop:semigroup covered by finitely many sumsets}
whether the commutative semigroup $[f(\mathbb{N}_{0})]$ generated
by the image of an integer-valued polynomial $f:\mathbb{N}_{0}\to\mathbb{Z}$
can be covered by finitely many sumsets $kf(\mathbb{N}_{0})$, and
explain in Proposition \ref{prop:semigroup not finitely covered by sumsets}
why a similar result does not generalize to a vector of polynomials
$f:\mathbb{N}_{0}\to\mathbb{Z}\times\cdots\times\mathbb{Z}$.

 The main result of this paper is given in Section \ref{sec:WP for H_=00007B2n+1=00007D(Z)}.
As a preparation, we generalize Kamke's key theorem and solve Waring's
problem for the abelian group $\mathbb{Z}^{m}$.  The fact that
symmetric polynomials in several variables with rational coefficients
can be written as a polynomial expression with rational coefficients
in the power sum symmetric polynomials enables us to connect the strategy
of iterated symmetrization with Kamke's result. To conquer the degenerate
case, we work with a finite product of affine translations of the
polynomial sequence. 

In section \ref{sec:WP for LNGs}, we discuss Waring's problem for
general locally nilpotent groups.  In particular, we prove a similar
result in Theorem (\ref{thm:WP for deg=00003D1}) for the case when
the degree of the polynomial sequence is exactly $1$ and a primitive
result in Theorem \ref{thm:WP for LNGs} for Waring's problem for
general locally nilpotent groups of degree $\ge2$, which is still
far from the question. 

\subsection*{Acknowledgment}

I thank my advisor Michael Larsen for the guidance and many helpful
discussions through this work. 

\section{Integer-valued Polynomials and Kamke's Generalization \label{sec:2 Integer-Valued Polynomials}}

\subsection{Integer-valued polynomials}

Let $D$ be a domain with quotient field $K$.  For any subset $E$
of $K$, 
\[
\Int(E,D):=\{f\in K[x]\mid f(E)\subseteq D\}
\]
is defined to be the set of $D$-valued polynomials on $E$. Clearly,
$\Int(E,D)$ is a subring of $K[x]$ and $\Int(D):=\Int(D,D)$ is
a $D$-module and a subring of $K[x]$ that contains $D[x]$.  Then
we have 
\begin{prop}
\label{prop:binomial coefficients polynomials} The polynomials $\binom{x}{n}$
with $n\in\mathbb{N}_{0}$ form a basis for the $\mathbb{Z}$-module
$\Int(\mathbb{Z})$.
\end{prop}
\begin{proof}
See \cite[Prop I.1.1]{Cahen_Chabert_1997}. 
\end{proof}
\begin{cor}
\label{cor: n+1 consecutive integers} A polynomial $f$ of degree
$n$ is integer-valued if and only if $f$ sends $n+1$ consecutive
integers to integers. In particular, for any $k\in\mathbb{Z}$, we
have 
\[
\Int(\mathbb{Z})=\Int(\mathbb{N}_{0},\mathbb{Z})=\Int(\mathbb{N},\mathbb{Z})=\Int(\mathbb{Z}_{\ge k},\mathbb{Z}).
\]
\end{cor}
\begin{proof}
See \cite[Cor I.1.2]{Cahen_Chabert_1997}.
\end{proof}
Given any set $P=\{x_{0},x_{2},\ldots,x_{n}\}\subset\mathbb{R}$,
the Lagrange basis polynomials associated with $P$ are 
\[
\ell_{i}(x;P):=\prod_{\substack{0\le k\le n\\
k\ne i
}
}\frac{x-x_{k}}{x_{i}-x_{k}}=\frac{(x-x_{0})}{(x_{i}-x_{0})}\cdots\frac{(x-x_{i-1})}{(x_{i}-x_{i-1})}\frac{(x-x_{i+1})}{(x_{i}-x_{i+1})}\cdots\frac{(x-x_{n})}{(x_{i}-x_{n})}.
\]
Since $x_{i}\in P$ are distinct, $l_{i}(x;P)$ is a polynomial of
degree $n$ such that $l_{i}(x_{j};P)=\delta_{ij}$ is the Kronecker
symbol. Then, any polynomial $f(x)$ of degree $\le n$ can be uniquely
written as 
\[
f(x)=\sum_{i=0}^{n}f(x_{i})l_{i}(x;P).
\]
In particular, for any integer $n\ge0$, let $P_{a,n}=\{a,a+1,\cdots,a+n\}$
be the set of $n+1$ consecutive integers starting from any integer
$a\in\mathbb{Z}$. 
\begin{prop}
\label{prop:Lagrange basis polynomials} The Lagrange basis polynomials
$l_{i}(x;P_{a,n})$ for $0\le i\le n$ form a basis for the $\mathbb{Z}$-submodule
of $\Int(\mathbb{Z})$ consisting of integer-valued polynomials of
degree at most $n$. 
\end{prop}
\begin{proof}
See \cite{Hensel1896} or \cite[Remark I.1.3]{Cahen_Chabert_1997}.
\end{proof}

\subsection{Kamke's key theorem and Kamke's generalization}

 According to Kamke \cite{Kamke1921}, the question about the simultaneous
decomposition of integers into powers of integers was raised by Hilbert
in a seminar. More precisely, Hilbert asked, for a given integer $n\ge2$,
under what as few restrictions on the positive integers $s_{1},\ldots,s_{n}$
as possible, there is a positive integer $N=N(n)$, such that for
each $n$ positive integers $s_{1},\ldots,s_{n}$, subject to those
restrictions, the system of equations 
\[
s_{1}=\sum_{\kappa=1}^{N}x_{\kappa},\quad s_{2}=\sum_{\kappa=1}^{N}x_{\kappa}^{2},\quad\ldots,\quad s_{n}=\sum_{\kappa=1}^{N}x_{\kappa}^{n}
\]
can be simultaneously solved by integers $x_{\kappa}\ge0$. To answer
this question and solve Waring's problem in integer-valued polynomials,
Kamke proved a theorem in two equivalent versions, which he called
Kernsatz. In the sequel, we will refer to it as Kamke's key theorem.

\begin{thm*}[Kamke's key theorem]
 For each integer $n\ge2$, there is an integer $N=N(n)>0$, an integer
$A>0$ and positive numbers $i_{1}$ and $i_{\nu}$, $J_{\nu}$ with
\[
0<i_{\nu}<J_{\nu}\qquad(\nu=2,3,\ldots,n),
\]
such that for each $n$ integers $s_{1},s_{2},\ldots,s_{n}$, divisible
by $A$ and subject to the following conditions 
\[
i_{1}<s_{1};\quad i_{\nu}s_{1}^{\nu}<s_{\nu}<J_{\nu}s_{1}^{\nu}\qquad(\nu=2,3,\ldots,n),
\]
the $n$ equations 
\[
s_{\nu}=\sum_{\kappa=1}^{N}x_{\kappa}^{\nu}\qquad(\nu=1,2,\ldots,n)
\]
are simultaneously solvable by integers $x_{\kappa}\ge0$. 
\end{thm*}
 We will refer to these conditions in Kamke's key theorem as Kamke
conditions and to the following unbounded open subsets as Kamke domains:
\begin{align*}
U(n,N) & =\{(s_{1},\cdots,s_{n})\in\mathbb{R}_{\ge0}^{n}\mid i_{1}<s_{1},i_{\nu}s_{1}^{v}<s_{v}<J_{v}s_{1}^{v},\nu=2,3,\ldots,n\}.
\end{align*}

\begin{example}
\label{exa:Kamke:n=00003D2} For $n=2$, Kamke's key theorem with
$N=5$, $A=2$, $i_{1}=7$, $i_{2}=\dfrac{1}{4}$, $J_{2}=\dfrac{1}{3}-\varepsilon$,
($0<\varepsilon<\dfrac{1}{12}$), is contained in \cite[Thm, page 2]{Kamke1921}. 
\end{example}
With the help of the key theorem, Kamke proved the following 
\begin{thm*}
Let $f$ be an integer-valued polynomial of degree $\ge2$ with a
positive leading coefficient. Then, there is an integer $N>0$ and
for each integer $Z\ge0$ there are integers \textup{$N'\ge0$, $N''\ge0$,
}$x_{1}\ge0,\ldots,x_{N'}\ge0$, such that $N'+N''\le N$ and 
\begin{align*}
Z & =\sum_{\kappa=1}^{N'}f(x_{\kappa})+N'';
\end{align*}
i.e., each integer $Z\ge0$ is decomposable into a bounded number
of elements in $f(\mathbb{N}_{0})$ with the addition of a bounded
number of units. 
\end{thm*}
\begin{rem*}
Kamke's original hypotheses include that $f(x)\ge0$ for each integer
$x\ge0$ and coefficients of $f$ are integers, both of which can
be removed, because if the leading coefficient is greater than $0$,
one always has $f(x)\ge0$ for sufficiently large $x\ge B$, so instead
of $f(x)$ one may consider its translation $f(x+B)$, and Proposition
\ref{prop:binomial coefficients polynomials} guarantees that the
coefficients are rational, and by multiplying with the least common
denominator $A$, one could work with $Af(x)$ instead of $f(x)$.

Notice that the theorem is trivial if $f(x)=kx+b$ with $k\in\mathbb{N}$
and $b\in\mathbb{Z}$. Indeed, for each integer $Z\ge0$ such that
$Z\ge\min\{f(\mathbb{N}_{0})\}$, there is an $x\in\mathbb{N}_{0}$
such that $f(x)\le Z<f(x+1)$.
\end{rem*}
In fact, Kamke's theorem follows easily from the following proposition,
which can be deduced from Kamke's key theorem. This was proved in
Kamke's paper but not stated explicitly. In the sequel, we will refer
to it as Kamke's non-explicit proposition. 
\begin{prop*}
Under Kamke's hypotheses, there exist $B,N'\in\mathbb{N}$ such every
sufficiently large number in $B\mathbb{N}$ can be written as a sum
of exactly $N'$ numbers in $f(\mathbb{N}_{0})$.
\end{prop*}
Thus, there exist positive integers $B,K,N'$, such that for all positive
integers $\zeta\ge K$, we have 
\[
B\zeta=\sum_{\kappa=1}^{N'}f(x_{\kappa}).
\]
Therefore, for all integers $Z\ge BK$, we have $Z=\sum_{\kappa=1}^{N'}f(x_{\kappa})+b$
with $0\le b<B$. Setting $N=BK+B+N'$, for each integer $Z>0$, we
obtain the following representation with $N''\ge0$:
\[
Z=\sum_{\kappa=1}^{N'}f(x_{\kappa})+N''\text{ with }N'+N''\le N.
\]

\subsection*{The semigroup generated by polynomial values}

The sumset $A+B$ (also known as the Minkowski sum) of two subsets
$A$ and $B$ of an ambient abelian group $(G,+)$ is defined to be
the set of all sums of one element from $A$ with the other element
from $B$.  Denote by $kA=A+A+\cdots+A$ the $k$-fold iterated sumset
of $A$ and denote by $[A]$ the commutative semigroup generated by
$A$. 

It is curious to know under what assumption on an integer-valued polynomial
$f:\mathbb{N}_{0}\to\mathbb{Z}$, the commutative semigroup $[f(\mathbb{N}_{0})]$
can be covered by finitely many sumsets $kf(\mathbb{N}_{0})$. What
if $f=(f_{1},\ldots,f_{n}):\mathbb{N}_{0}\to\mathbb{Z}^{n}$ is a
vector of integer-valued polynomials $f_{i}:\mathbb{N}_{0}\to\mathbb{Z}$?
Answers to these questions are the starting point of this work. 

Clearly, we have the following ascending chain of unions of sumsets
\begin{equation}
A\subseteq\bigcup_{k=1}^{2}kA\subseteq\cdots\subseteq\bigcup_{k=1}^{m}kA\subseteq\cdots\subseteq[A].\label{eq:chain}
\end{equation}
If it does not stabilize for $m$ large enough, then $\bigcup_{k=1}^{m}kA$
and thus $kA$ for all $1\le k\le m$ are always proper subsets of
$[A]$. Further, if $A$ contains the zero element $0$ of $G$, then
$[A]=\bigcup_{k=1}^{\infty}kA$ is the commutative monoid generated
by $A$ and $kA\subseteq(k+1)A$ for any $k\in\mathbb{N}$ and (\ref{eq:chain})
becomes 
\begin{equation}
\{0\}\subseteq A\subseteq2A\subseteq\cdots\subseteq kA\subseteq\cdots\subseteq[A].\label{eq:chain 0}
\end{equation}
In this case, $[A]=kA$ for some $k\in\mathbb{N}$ if and only if
the ascending chain stabilizes for $k$ large enough.

Notice that $1$ (resp. $1$ and $0$, resp. $1$ and $-1$) can be
represented by a sum of finitely many numbers in $f(\mathbb{N}_{0})$
if and only if $\mathbb{N}\subseteq[f(\mathbb{N}_{0})]$ (resp. $\mathbb{N}_{0}\subseteq[f(\mathbb{N}_{0})]$,
resp. $\mathbb{Z}\subseteq[f(\mathbb{N}_{0})]$), and if this is the
case, then Kamke's result implies that each integer $Z>0$ (resp.
$Z\ge0$, resp. $Z\ge0$) is decomposable into a bounded number of
numbers in $f(\mathbb{N}_{0})$. Therefore, we have the following 
\begin{cor}[of Kamke's theorem]
\label{cor:of Kamke's theorem}  If $f:\mathbb{N}_{0}\to\mathbb{Z}$
is a polynomial of degree $n\ge1$ with a positive leading coefficient,
and $\mathbb{N}\subseteq[f(\mathbb{N}_{0})]$ (resp. $\mathbb{N}_{0}\subseteq[f(\mathbb{N}_{0})]$),
then each $Z>0$ (resp. $Z\ge0$) is decomposable into a bounded number
of elements in $f(\mathbb{N}_{0})$, that is, there exists a uniform
$N\in\mathbb{N}$ such that $\mathbb{N}\subseteq\bigcup_{k=1}^{N}kf(\mathbb{N}_{0})$
(resp. $\mathbb{N}_{0}\subseteq Nf(\mathbb{N}_{0})$). 
\end{cor}
The following proposition characterizes integer-valued polynomials
$f:\mathbb{N}_{0}\to\mathbb{Z}$ such that the commutative semigroup
$[f(\mathbb{N}_{0})]$ can be covered by finitely many sumsets $kf(\mathbb{N}_{0})$. 
\begin{prop}
\label{prop:semigroup covered by finitely many sumsets} Let $f:\mathbb{N}_{0}\to\mathbb{Z}$
be an integer-valued polynomial of degree $d$. Then, we have 
\[
[f(\mathbb{N}_{0})]=\bigcup_{k=1}^{N}kf(\mathbb{N}_{0}),\text{ for some }N=N(f)\in\mathbb{N},
\]
if and only if either $d=-\infty$, i.e., $f=0$, or $d\ge1$ and
either $f(\mathbb{N}_{0})\subseteq\mathbb{N}_{0}$ or $f(\mathbb{N}_{0})\subseteq-\mathbb{N}_{0}$. 
\end{prop}
To prove this proposition, we need to know when the greatest common
divisor of $f(\mathbb{N}_{0})$ is $1$. The condition $\gcd f(\mathbb{N}_{0})=1$
is equivalent to that for each prime number $p$, there exists $n\in\mathbb{N}_{0}$,
dependent on $p$, such that $f(n)\not\equiv0\pmod p$. Since $f(n)\equiv f(n\mod p)\mod p$,
we can choose a sufficiently large $n$ among the equivalence class
$n\mod p$. Hence, we have 
\begin{lem}
For any integer-valued polynomial $f\in\Int(\mathbb{Z})$, the following
conditions are all equivalent: 
\[
\gcd f(\mathbb{Z})=1\Leftrightarrow\gcd f(\mathbb{N}_{0})=1\Leftrightarrow\gcd f(\mathbb{N})=1\Leftrightarrow\gcd f(\mathbb{Z}_{\ge k})\text{ for any }k\in\mathbb{Z}.
\]
\end{lem}
Furthermore, the following three lemmas give shortcuts to verify these
conditions:
\begin{lem}
\label{lem:gcd formula binomial} By Proposition \ref{prop:binomial coefficients polynomials},
any integer-valued polynomial $f(x)$ of degree $d\ge0$ can be written
as 
\[
f(x)=a_{d}\binom{x}{d}+\cdots+a_{1}\binom{x}{1}+a_{0}\binom{x}{0}
\]
with $a_{i}\in\mathbb{Z}$ for all $i=0,\cdots,d$. Then, for any
$k\in\mathbb{Z}$, we have 
\[
\gcd f(\mathbb{Z})=\gcd f(\mathbb{N}_{0})=\gcd f(\mathbb{N})=\gcd f(\mathbb{Z}_{\ge k})=\gcd\{a_{0},a_{1},\cdots,a_{d}\}.
\]
\end{lem}
\begin{proof}
Trivial. 
\end{proof}
\begin{lem}
\label{lem:gcd formula Lagrange} Let $P_{a,d}=\{a,a+1,\cdots,a+d\}$
be the set of $d+1$ consecutive integers starting from any integer
$a\in\mathbb{Z}$. For any integer-valued polynomial $f(x)$ of degree
$d\ge0$ and any $k\in\mathbb{Z}$, we have 
\[
\gcd f(\mathbb{Z})=\gcd f(\mathbb{N}_{0})=\gcd f(\mathbb{N})=\gcd f(\mathbb{Z}_{\ge k})=\gcd f(P_{a,d}).
\]
\end{lem}
\begin{proof}
By Proposition \ref{prop:Lagrange basis polynomials}. 
\end{proof}
\begin{lem}
\label{lem:gcd formula via CRT} For any integer-valued polynomial
$f(x)$ of degree $d\ge0$ and any $k\in\mathbb{Z}$, we have 
\[
\gcd f(\mathbb{Z})=\gcd f(\mathbb{N}_{0})=\gcd f(\mathbb{N})=\gcd f(\mathbb{Z}_{\ge k})=1,
\]
if and only if there exist $m_{1},m_{2}\in\mathbb{Z}_{\ge k}$ such
that $\gcd\{f(m_{1}),f(m_{2})\}=1$. 
\end{lem}
\begin{proof}
By Chinese remainder theorem. 
\end{proof}
For any finite number of vectors $v_{1},v_{2},\ldots,v_{n}$ in $\mathbb{R}^{N}$,
a vector of the form $\alpha_{1}v_{1}+\alpha_{2}v_{2}+\cdots+\alpha_{n}v_{n}$
with $\alpha_{i}\in\mathbb{R}_{\ge0}$ is called a conical sum of
$v_{1},v_{2},\ldots,v_{n}$, and is called an integer conical sum
if in addition $\alpha_{i}\in\mathbb{N}_{0}$.  Given a set $S=\{a_{1},a_{2},\ldots,a_{n}\}$
of positive integers such that $\gcd S=1$, the Diophantine Frobenius
problem asks to determine the largest integer that cannot be expressed
as an integer conical sum of these numbers. This integer, usually
denoted by $g(a_{1},a_{2},\ldots,a_{n})$ or $g(S)$, is called the
\emph{Frobenius number} of the set $S$.  The following folk result
(see \cite[Thm 1.0.1]{Alfonsin2005}) guarantees the finiteness of
this number:
\begin{thm*}
Given a set $S=\{a_{1},a_{2},\ldots,a_{n}\}$ of positive integers,
if $\gcd S=1$, then there exists an integer $N$ such that any integer
$s\ge N$ is representable as an integer conical sum of $a_{1},\ldots,a_{n}$.
\end{thm*}
 Now we are ready to prove Proposition \ref{prop:semigroup covered by finitely many sumsets}. 
\begin{proof}
If $d=-\infty$, i.e., $f=0$, then $N=1$ and $[f(\mathbb{N}_{0})]=f(\mathbb{N}_{0})$;
if $d=0$, i.e., $f(\mathbb{N}_{0})=\{a\}$ for some $0\ne a\in\mathbb{Z}$,
then $[f(\mathbb{N}_{0})]=a\mathbb{N}$ cannot be covered by finitely
many sumsets. 

Without loss of generality, we may assume that $f:\mathbb{N}_{0}\to\mathbb{Z}$
is a polynomial of degree $d\ge1$ with a positive leading coefficient.
There exists $b\in\mathbb{N}$ such that $f(x)>0$ for all $x>b$,
so $f(x)$ achieves its minimum value $B\in\mathbb{Z}$ in $[0,b]\cap\mathbb{N}_{0}$.
If $B<0$, then $B\mathbb{N}\subset[f(\mathbb{N}_{0})]$ cannot be
covered by finitely many sumsets. So we only need to consider the
case when $B\ge0$, i.e., when $f(\mathbb{N}_{0})\subseteq\mathbb{N}_{0}$.
If necessary, we can multiply $f$ by a nonzero rational number and
assume that $\gcd f(\mathbb{N}_{0})=1$. By the discussion above,
the condition on $f$ is necessary, and it suffices to prove the assertion
when $d\ge1$ such that $f(\mathbb{N}_{0})\subseteq\mathbb{N}_{0}$
and $\gcd f(\mathbb{N}_{0})=1$. 

If $d=1$, then by Lemma \ref{lem:gcd formula binomial} or \ref{lem:gcd formula Lagrange},
we must have $f(x)=ax+b$, where $a\in\mathbb{N}$ and $b\in\mathbb{N}_{0}$
and $\gcd\{b,a+b\}=\gcd f(\mathbb{N}_{0})=1$. If $a=1$, then $[f(\mathbb{N}_{0})]=\mathbb{Z}_{\ge b}=f(\mathbb{N}_{0})$
and $N=1$. If $b=0$, then $a=1$, thus $[f(\mathbb{N}_{0})]=\mathbb{N}_{0}=f(\mathbb{N}_{0})$
and $N=1$. If $b=1$, then $[f(\mathbb{N}_{0})]=\mathbb{N}=\bigcup_{k=1}^{a}kf(\mathbb{N}_{0})$
and $N=a$. So we may assume $a\ne1$ and $b\ne0,1$. 

Each integer greater than the Frobenius number $g(a+b,b)$ belongs
to the semigroup $[f(\mathbb{N}_{0})]$, since $a+b,b\in f(\mathbb{N}_{0})$.
Then, from each equivalence class $i+a$ of $\mathbb{Z}/a$, where
$i\in\{0,1,\ldots,a-1\}$, we can pick the smallest integer $Z_{i}\in[f(\mathbb{N}_{0})]$.
Then, let $N_{i}$ be the smallest number such that each $Z_{i}$
is decomposed into the sum of exactly $N_{i}$ elements of $f(\mathbb{N}_{0})$.
Let $M=\max\{Z_{i}\mid0\le i\le a-1\}$. Then, for any $Z\ge M+b$,
there exists exactly one way to represent 
\[
Z=aK+b+Z_{i},\text{ for some }i.
\]
From $aK+b+Z_{i}\ge M+b$, we obtain $aK\ge M-Z_{i}\ge0$ and thus
$K\ge0$. Hence, $Z$ can be decomposed into the sum of exactly $1+N_{i}$
elements of $f(\mathbb{N}_{0})$. Let $N'$ be the least integer such
that each $Z<M+b$ in $[f(\mathbb{N}_{0})]$ can be decomposed into
the sum of at most $N'$ elements of $f(\mathbb{N}_{0})$. Then, we
can choose $N=\max\{N',1+N_{i}\mid i=0,1,\ldots,a-1\}$. 

For $d\ge2$, a similar proof follows from Kamke's non-explicit proposition.
\end{proof}
 But in general Proposition \ref{prop:semigroup covered by finitely many sumsets}
does not generalize to a vector of polynomials $f:\mathbb{N}_{0}\to\mathbb{Z}\oplus\cdots\oplus\mathbb{Z}$.
The failure has to do with the boundary of the convex cone spanned
by vectors of the polynomial values. Here is an example when it happens.

\begin{example*}
Consider the following example 
\[
f:\mathbb{N}_{0}\to\mathbb{Z}\oplus\mathbb{Z}\oplus\cdots\oplus\mathbb{Z};\quad x\mapsto(x,x^{2},\ldots,x^{n}),
\]
which is a vector of polynomials of degree $\le n$. Since $f(0)=(0,0,\ldots,0)$,
we have the following ascending chain of sumsets 
\[
f(\mathbb{N}_{0})\subseteq2f(\mathbb{N}_{0})\subseteq\cdots\subseteq kf(\mathbb{N}_{0})\subseteq\cdots\subseteq[f(\mathbb{N}_{0})].
\]
But this chain does not stabilize for $k$ large enough.  Indeed,
we have 
\[
(N+1,N+1,\ldots,N+1)\in(N+1)f(\mathbb{N}_{0}),
\]
which can only be the sum of $(1,1,\ldots,1)$'s and $(0,0,\ldots,0)$'s,
and we need at least $(N+1)$ $(1,1,\ldots,1)$'s. Hence, $kf(\mathbb{N}_{0})$
is always a proper subset of $[f(\mathbb{N}_{0})]$ for any $k\in\mathbb{N}$. 
\end{example*}
More generally, motivated by the above example, we have the following
result. 
\begin{prop}
\label{prop:semigroup not finitely covered by sumsets} Consider the
following map of a vector of polynomials 
\[
f:\mathbb{N}_{0}\to\mathbb{Z}\oplus\mathbb{Z}\oplus\cdots\oplus\mathbb{Z};\quad x\mapsto(f_{1}(x),f_{2}(x),\ldots,f_{n}(x)),
\]
where at least two $f_{i}:\mathbb{N}_{0}\to\mathbb{Z}$ have degree
$\ge0$ and at least two $f_{j}:\mathbb{N}_{0}\to\mathbb{Z}$ are
not proportional to each other. Then, for any $m\in\mathbb{N}$, $\bigcup_{k=1}^{m}kf(\mathbb{N}_{0})$
and thus $kf(\mathbb{N}_{0})$ for all $1\le k\le m$ are always proper
subsets of $[f(\mathbb{N}_{0})]$.
\end{prop}
\begin{proof}
 The proof is essentially done by contradiction and the strategy
has been illustrated in the previous example and Proposition \ref{prop:semigroup covered by finitely many sumsets}. 
\end{proof}
\begin{rem*}
Consider the ring $\mathcal{O}$ of integers of a totally imaginary
number field $K$, and a closed commutative $\mathcal{O}$-subscheme
$\mathcal{G}$ of the group scheme $\mathcal{U}_{k}$ of unitary $k\times k$
matrices, and a generating $\mathcal{O}$-morphism $f:\mathbb{A}^{1}\to\mathcal{G}$,
i.e., $f$ is generating as a $K$-morphism. Then, \cite[Thm 3.1 or Prop 3.6]{LN2019}
implies that there exists some positive integer $M$ such that $Mf(\mathcal{O})$
is a subgroup of finite index in $\mathcal{G}(\mathcal{O})$. But
if we are working in a real number field $K$, then Proposition \ref{prop:semigroup not finitely covered by sumsets}
implies that an analogous statement will no longer hold in general,
for example, when $K=\mathbb{Q}$, $\mathcal{O}=\mathbb{Z}$, $\mathcal{G}=\mathbb{A}^{2}$
as a closed commutative $\mathbb{Z}$-subscheme of $\mathcal{U}_{3}$
and $f:\mathbb{A}^{1}\to\mathbb{A}^{2}$ defined by $x\mapsto(x^{2},x^{4})$. 
\end{rem*}

\section{Waring's Problem for General Discrete Heisenberg Groups $H_{2n+1}(\mathbb{Z})$
\label{sec:WP for H_=00007B2n+1=00007D(Z)}}

The main result of this paper will be given in this section. But let
us first recall definitions of polynomial maps, polynomial sequences
and polynomial sets. For any map $f:S\to G$ from a nonempty semigroup
$S$ to a group $G$ and for any $s\in S$, we define the following
left and right forward finite differences 
\begin{align*}
L_{s}(f):S & \rightarrow G &  &  & R_{s}(f):S & \rightarrow G\\
t & \mapsto f(s+t)f(t)^{-1}, &  &  & t & \mapsto f(t)^{-1}f(s+t).
\end{align*}
Then, $f:S\to G$ is a called a polynomial map of degree $\le d$,
if for any $s_{1},s_{2},\ldots,s_{d+1}\in S$, 
\[
D_{s_{1}}D_{s_{2}}\cdots D_{s_{d+1}}f\equiv1_{G},
\]
where each $D$ is arbitrarily taken to be $L$ or $R$. The minimal
$d$ with this property is called the degree of $f$. In particular,
a map $f:S\to G$ is a polynomial map of degree $-\infty$ if $f$
maps $S$ to the identity element $1_{G}$ of $G$, and $f$ is a
polynomial map of degree $0$ if it is a constant $\neq1_{G}$. A
polynomial sequence is a polynomial map of the form $f:\mathbb{N}_{0}\to G$. 

A polynomial set $U=\Ima f$ in a path-connected nilpotent Lie group
$N$ is given by the image of a continuous polynomial map $f:\mathbb{R}_{\ge0}^{n}\to N$
for some $n$, and $U$ is called open (resp. closed, resp. proper)
if $U$ is open (resp. is closed, resp. has nonempty interior). A
nonempty subset $V$ of a nilpotent group $G$ is called a polynomial
set, if it is the inverse image $\phi^{-1}(U)$ of a polynomial set
$U$ of a nilpotent Lie group $N$ along some group homomorphism $\phi\colon G\to N$,
and $V=\phi^{-1}(U)$ is called open (resp. closed, resp. proper)
in $G$ provided that $U$ has the same property in $N$. In this
work, generalized cones in nilpotent groups are given by proper polynomial
sets. 

\subsection{Case in which $\langle g\rangle\subset\mathbb{Z}^{m}$}

The following theorem is a generalization of Kamke's key theorem. 
\begin{thm}
\label{thm:WP for Z^m} Let $g:\mathbb{N}_{0}\to\mathbb{Z}^{m}$ be
a polynomial sequence. If $g\mod N$ is non-constant for any subgroup
$N$ of infinite index in $\mathbb{Z}^{m}$, then there exists $M\in\mathbb{N}$,
a subgroup $H$ of finite index in $\mathbb{Z}^{m}$ and a proper
polynomial set $V$ of $H$ such that each element in $V$ can be
written as a sum of exactly $M$ elements in $g(\mathbb{N}_{0})$. 
\end{thm}
\begin{rem*}
In this case, $g=(g_{1},\ldots,g_{m})$ is just a vector of polynomials
$g_{i}:\mathbb{N}_{0}\to\mathbb{Z}$ and $V=H\cap\Ima f$ is the intersection
of $H$ with the image of some proper continuous polynomial map $f=(f_{1},\ldots,f_{m}):\mathbb{R}_{\ge0}^{n}\to\mathbb{R}^{m}$
along the embedding $\varphi:H\hookrightarrow\mathbb{R}^{m}$, where
each $f_{i}:\mathbb{R}_{\ge0}^{n}\to\mathbb{R}$ is a polynomial in
$n$ variables; cf. \cite[Cor 13 and Cor 17]{Hu2020}. 
\end{rem*}
\begin{rem*}
Proposition \ref{prop:semigroup not finitely covered by sumsets}
implies that $[g]$ in general cannot be covered by finitely many
sumsets $kg(\mathbb{N}_{0})$. So the best possible result is that
points of certain proper polynomial set of a finite index subgroup
$H$ of $\langle g\rangle$ can be covered by finitely many sumsets
$kg(\mathbb{N}_{0})$.
\end{rem*}
\begin{proof}
Notice that $h=g-g(0)$ is another polynomial sequence satisfying
our hypothesis: for any subgroup $N$ of infinite index in $\mathbb{Z}^{m}$,
if $h\mod N$ is a constant, then so is $g\mod N$. So we may assume
that $g(0)=(0,\ldots,0)$. Indeed, suppose we have proven the theorem
for $h$ by finding the desirable positive integer $M$, subgroup
$H$ of finite index in $\mathbb{Z}^{m}$ and proper polynomial set
$V$ of $H$. Then, we have $\sum_{i=1}^{M}g(x_{i})-\sum_{i=1}^{M}h(x_{i})=Mg(0)$.
Since $h(0)=(0,\ldots,0)$, we can find a larger $M$ so that $Mg(0)$
lies in the lattice $H$ and translate the proper polynomial set $V$
of $H$ by a certain scalar multiple of $g(0)$ to obtain a proper
polynomial set that works for $g$. Moreover, since $g\mod N$ is
non-constant for any subgroup $N$ of infinite index in $\mathbb{Z}^{m}$,
$\langle g\rangle$ must have finite index in $\mathbb{Z}^{m}$. Then,
there exists a nonsingular matrix $D\in M(m,\mathbb{Z})$ such that
$\langle g\rangle=D\mathbb{Z}^{m}$. Replacing $g$ by $D^{-1}g$
if necessary, we may also assume that $\langle g\rangle\cong\mathbb{Z}^{m}$.

Write $g(x)=(g_{1}(x),\ldots,g_{m}(x))\in\mathbb{Z}^{m}$, where $g_{i}(x)=\sum_{j=1}^{d}c_{ji}x^{j}$,
$d=\max\{\deg g_{i}:1\le i\le m\}$, and $c_{ji}\in\mathbb{Q}$. Since
$\mathbb{Z}^{m}$ is generated by $g$, there exist $M_{i}\in\mathbb{N}$,
$\varepsilon_{ki}\in\{\pm1\}$ and $x_{ki}\in\mathbb{N}_{0}$ such
that 
\begin{align*}
e_{i} & =\left(\sum_{k=1}^{M_{i}}\varepsilon_{ki}g_{1}(x_{ki}),\ldots,\sum_{k=1}^{M_{i}}\varepsilon_{ki}g_{m}(x_{ki})\right),
\end{align*}
where $e_{1}=(1,0,\ldots,0),\ldots,e_{m}=(0,\ldots,0,1)$ are the
standard generators of $\mathbb{Z}^{m}$. Rewrite the above equation
as follows: 
\[
e_{i}=\left(\sum_{k=1}^{M_{i}}\varepsilon_{ki}x_{ki},\ldots,\sum_{k=1}^{M_{i}}\varepsilon_{ki}x_{ki}^{d}\right)\begin{pmatrix}c_{11} & \cdots & c_{1m}\\
\vdots & \ddots & \vdots\\
c_{d1} & \cdots & c_{dm}
\end{pmatrix}.
\]
For this system of linear equations to be solvable, the coefficient
matrix 
\[
C=\begin{pmatrix}c_{11} & \cdots & c_{1m}\\
\vdots & \ddots & \vdots\\
c_{d1} & \cdots & c_{dm}
\end{pmatrix}\in M_{d,m}(\mathbb{Q})
\]
must have full column rank $m$ and thus $d\ge m$. 

If $m=1$, then Kamke's non-explicit proposition covers the case in
which $d\ge2$; If $m=1=d$, then $g(x)=\pm x$ gives rise to a trivial
case. So we may assume that $m\ge2$ and thus $d\ge2$. Let $M$ be
$N(d)$ as in Kamke's key theorem and $\tilde{g}:\mathbb{N}_{\ge0}^{M}\to\mathbb{Z}^{m}$
be defined by 
\[
(x_{1},\ldots,x_{M})\mapsto\sum_{k=1}^{M}g(x_{k})=\left(\sum_{k=1}^{M}g_{1}(x_{k}),\ldots,\sum_{k=1}^{M}g_{m}(x_{k})\right).
\]
Then, we can write 
\[
\tilde{g}(x_{1},\ldots,x_{M})=\left(\sum_{j=1}^{d}c_{j1}s_{j},\ldots,\sum_{j=1}^{d}c_{jm}s_{j}\right)=(s_{1},\ldots,s_{d})C,
\]
where $s_{j}=\sum_{k=1}^{M}x_{k}^{j}$ for $j=1,\ldots,d$. Let $A$
be as in Kamke's key theorem. Replace $A$ by $kA$ for some $k\in\mathbb{N}$
if necessary so that $AC\in M_{d,m}(\mathbb{Z})$. Now we define a
linear polynomial map 
\[
p:\mathbb{R}_{\ge0}^{d}\to\mathbb{R}^{m};(s_{1},\ldots,s_{d})\mapsto(s_{1},\ldots,s_{d})C
\]
By \cite[Lem 8]{Hu2020}, a Kamke domain always contains a proper
polynomial set given by some continuous polynomial map $q:\mathbb{R}_{\ge0}^{n}\to\mathbb{R}^{d}$.
Let $f=p\circ q$. Then, $H=p(A\mathbb{Z}^{d})=D\mathbb{Z}^{m}$ is
a lattice and thus a finite index subgroup in $\mathbb{Z}^{m}$ for
some nonsingular $D\in M(m,\mathbb{Z})$. Let $U(d,M)$ be the Kamke
domain. Then, we have the following relation of sets 
\[
\tilde{g}(\mathbb{N}_{\ge0}^{M})\supset p(U(d,M)\cap A\mathbb{N}_{0}^{d})=p(U(d,M)\cap A\mathbb{Z}^{d})\supset\Ima f\cap H.
\]
Let $\phi:H\hookrightarrow\mathbb{R}^{m}$ be the inclusion map. Then,
$U=\Ima f$ (resp. $V=U\cap H=\phi^{-1}(U)$) is a proper polynomial
subset in $\mathbb{R}^{m}$ (resp. $H$) and each element of $V$
can be written as a sum of exactly $M$ elements in the sequence $g_{0},g_{1},g_{2},\ldots$. 
\end{proof}
\begin{rem*}
Notice that each entry $c_{ji}$ the coefficient matrix $C$ can be
thought as $g_{i}^{(j)}(0)/j!$. So it is natural to consider the
following matrix associated with $g=(g_{1},\ldots,g_{m})$: 
\[
\mathbf{J}(g)(x)=\begin{pmatrix}g_{1}^{(1)}(x) & \cdots & g_{m}^{(1)}(x)\\
\vdots & \ddots & \vdots\\
g_{1}^{(d)}(x) & \cdots & g_{m}^{(d)}(x)
\end{pmatrix}=\diag(1!,2!,\ldots,d!)C.
\]
Then, $C$ and $\mathbf{J}(g)(0)$ have the same rank $m$. Moreover,
for any pair $(a,b)\in\mathbb{N}\times\mathbb{N}_{0}$, by the chain
rule, the associated matrix of $g(ax+b)$ satisfies the following
equation
\[
\mathbf{J}(g(ax+b))(0)=\diag(a,a^{2},\ldots,a^{d})\mathbf{J}(g)(b).
\]
If $g$ satisfies the hypothesis of Theorem \ref{thm:WP for Z^m},
then so does $g(ax+b)$. Then, the same argument shows $\mathbf{J}(g(ax+b))(0)$
and thus $\mathbf{J}(g)(b)$ have the same rank $m$. 
\end{rem*}

\subsection{A simple observation}

 Consider a polynomial sequence of the following form 
\[
g=\begin{pmatrix}1 & g_{1,2} & g_{1,3} & \cdots & g_{1,n}\\
 & 1 & g_{2,3} & \cdots & g_{2,n}\\
 &  & 1 & \ddots & \vdots\\
 &  &  & \ddots & g_{n-1,n}\\
 &  &  &  & 1
\end{pmatrix}:\mathbb{N}_{0}\rightarrow\mathcal{U}_{n}(\mathbb{Z}).
\]
By \cite[Thm 11]{Hu2020}, each $g_{i,j}$\index{$g_{i,j}$, the $(i,j)$-entry of the polynomial sequence $g$ in
a matrix form} is a polynomial of degree $d_{i,j}$\index{$d_{i,j}$, the degree of the polynomial $g_{i,j}$}
with rational coefficients. Let $L$\index{Linmathbb{N}@$L\in\mathbb{N}$}
be an arbitrary natural number and $x_{1},x_{2},\ldots,x_{L}$ be
variables in $\mathbb{N}_{0}$. Consider the ordered product\index{hat{g}:(x_{1},ldots,x_{L})mapsto g(x_{1})cdots g(x_{L}), the 
ordered product@$\hat{g}:(x_{1},\ldots,x_{L})\mapsto g(x_{1})\cdots g(x_{L})$, the
ordered product} 
\[
\hat{g}:=\bigodot_{i=1}^{L}g:\mathbb{N}_{0}^{L}\to\mathcal{U}_{n}(\mathbb{Z});\quad(x_{1},\ldots,x_{L})\mapsto g(x_{1})\cdots g(x_{L}).
\]
Let $\hat{d}_{i,j}(L)$\index{hat{d}_{i,j}, the degree of the polynomial map hat{g}_{i,j}@$\hat{d}_{i,j}$, the degree of the polynomial map $\hat{g}_{i,j}$}
be the total degree of the polynomial map $\hat{g}_{i,j}(x_{1},\ldots,x_{L})$\index{hat{g}_{i,j}, the (i,j)-entry of the polynomial map hat{g} in a 
matrix form@$\hat{g}_{i,j}$, the $(i,j)$-entry of the polynomial map $\hat{g}$
in a matrix form} in the $L$ variables. Then, the set $\{\hat{d}_{i,j}(L)\mid1\le i<j\le n,L\in\mathbb{N}\}$
has an upper bound $B$\index{Binmathbb{N}, an upper bound of the total degrees@$B\in\mathbb{N}$, an upper bound of the total degrees},
which satisfies:
\begin{align}
B\ge & \max_{1\le i<j\le n}\left\{ d_{i,j},d_{i,l}+d_{l,j},\cdots,d_{i,l_{1}}+d_{l_{1},l_{2}}+\cdots+d_{l_{m},j}\right|\label{eq:inequality B}\\
 & \left.i+1\le l\le j-1,\cdots,i+1\le l_{1}<l_{2}<\cdots<l_{m}\le j-1\right\} .\nonumber 
\end{align}
Indeed, each entry $\hat{g}_{i,j}(x_{1},\ldots,x_{L})$ is of the
form 
\begin{align*}
 & \sum_{1\le k\le L}g_{i,j}(x_{k})+\sum_{1\le k_{1}<k_{2}\le L}\left(\sum_{i+1\le l\le j-1}g_{il}(x_{k_{1}})g_{lj}(x_{k_{2}})\right)+\cdots\\
+ & \sum_{1\le k_{1}<\cdots<k_{m}\le L}\left(\sum_{i+1\le l_{1}<l_{2}<\cdots<l_{m}\le j-1}g_{il_{1}}(x_{k_{1}})g_{l_{1}l_{2}}(x_{k_{2}})\cdots g_{l_{m}j}(x_{k_{m}})\right)+\cdots.
\end{align*}
Since the corresponding terms in the above equation vanish for $m\ge j-i$,
the total degree of each entry $\hat{g}_{i,j}$ has a universal upper
bound $B$, which satisfies (\ref{eq:inequality B}). 

\subsection{Case in which $\langle g\rangle\subset H_{2n+1}(\mathbb{Z})$}

Recall that for any $n\times n$ upper unitriangular matrix $X$ in
the Lie group $\mathcal{U}_{n}(\mathbb{R})$, its logarithm is defined
by 
\[
\log X:=\sum_{k=1}^{\infty}(-1)^{k+1}\dfrac{(X-I)^{k}}{k}=\sum_{k=1}^{n-1}(-1)^{k}\dfrac{(X-I)^{k}}{k},
\]
which lies in the nilpotent Lie algebra $\mathfrak{g}$ of strictly
upper triangular matrices. In this case, the logarithm and the exponential
\[
\exp(Y):=\sum_{k=0}^{\infty}\dfrac{Y^{k}}{k!}=\sum_{k=0}^{n-1}\dfrac{Y^{k}}{k!},\qquad\forall Y\in\mathfrak{g}
\]
are mutually inverse to each other. 

As illustrated in the following proposition, the strategy of iterated
symmetrization will be applied to solve Waring's problem in general
discrete Heisenberg groups. 
\begin{prop}
\label{prop:symmetrization} Let $X_{1},\ldots,X_{N}$ be any matrices
in the continuous Heisenberg group $H_{2n+1}(\mathbb{R})$ for some
integer $n\ge1$. Then, there exists $M\in\mathbb{N}$ and a sequence
$a_{1},a_{2},\ldots,a_{M}\in\{1,2,\ldots,N\}$, only dependent on
$N$, such that 
\begin{equation}
X_{a_{1}}X_{a_{2}}\cdots X_{a_{M}}=\exp\left(\frac{M}{N}\sum_{i=1}^{N}\log X_{i}\right).\label{eq:sym eq in H_2n+1(R)}
\end{equation}
In particular, each $a_{i}\in\{1,2,\ldots,N\}$ appears exactly $M/N$
times. 
\end{prop}
\begin{proof}
If we set $Y_{i}=\log X_{i}$ for $1\le i\le N$, then it is equivalent
to proving 
\[
\exp(Y_{a_{1}})\exp(Y_{a_{2}})\cdots\exp(Y_{a_{M}})=\exp\left(\frac{M}{N}\sum_{i=1}^{N}Y_{i}\right)
\]
for some sequence $a_{1},a_{2},\ldots,a_{M}\in\{1,2,\ldots,N\}$.
From the Baker-Campbell-Hausdorff formula $\exp(X)\exp(Y)=\exp(X+Y+\frac{1}{2}[X,Y])$,
we obtain that 
\begin{align*}
\exp(Y_{a_{1}})\exp(Y_{a_{2}})\cdots\exp(Y_{a_{M}}) & =\exp\left(\sum_{i=1}^{M}Y_{a_{i}}+\dfrac{1}{2}\sum_{1\le i<j\le M}[Y_{a_{i}},Y_{a_{j}}]\right).
\end{align*}
Hence, if we choose the sequence to be $1,2,\ldots,N-1,N,N,N-1,\ldots,2,1$,
then the main terms $\sum_{i=1}^{M}Y_{a_{i}}$ becomes $\exp\left(2\sum_{i=1}^{N}Y_{i}\right)$
and all the error terms, i.e., the commutators, cancel with each other,
since $[Y_{i},Y_{j}]+[Y_{j},Y_{i}]=0$ for all $1\le i,j\le n$. 
\end{proof}
 From \cite[Theorem 14]{Hu2020} or the proposition above, we obtain
the following
\begin{cor}
Let $\hat{g}=\bigodot_{i=1}^{N}g:\mathbb{R}_{\ge0}^{N}\rightarrow H_{2n+1}(\mathbb{R})$
be the ordered product obtained from any continuous polynomial map
$g:\mathbb{R}_{\ge0}\to H_{2n+1}(\mathbb{R})$. Then, there is a finite
natural number $M$, only dependent on $N$, and a sequence $\sigma_{1},\sigma_{2},\ldots,\sigma_{M}$
in the permutation group $S_{N}$ such that the product 
\[
\tilde{g}=\prod_{i=1}^{M}\sigma_{i}(\hat{g})=\sigma_{1}(\hat{g})\sigma_{2}(\hat{g})\cdots\sigma_{M}(\hat{g})
\]
is a symmetric continuous polynomial map of the form 
\[
\tilde{g}(x_{1},\ldots,x_{N})=\exp\left(M\sum_{i=1}^{N}\log g(x_{i})\right).
\]
\end{cor}

\subsection*{Proof of the main result}

For any $n\in\mathbb{N}$, we denote by $G:=H_{2n+1}(\mathbb{Z})$\index{G:=H_{2n+1}(mathbb{Z}), the discrete Heisenberg group@$G:=H_{2n+1}(\mathbb{Z})$, the discrete Heisenberg group}
the general discrete Heisenberg group. Let $D$ be a nonnegative integer,
$\Gamma$ be a subgroup of $G$ and 
\[
\pi_{D}:\Gamma\to G\to H_{2n+1}(\mathbb{Z}/D)
\]
be the reduction homomorphism. The \emph{principal congruence subgroup}
of level $D$ in $\Gamma$ is defined to be the kernel of $\pi_{D}$,
and is usually denoted by $\Gamma(D)$ and a congruence subgroup of
$\Gamma$ is defined to be any group containing a principal congruence
subgroup. It is immediate that $\Gamma(D)$ is a normal subgroup of
finite index in $\Gamma$ and all congruence subgroups containing
$\Gamma(D)$ have finite index in $\Gamma$ and correspond to subgroups
of the finite groups $\pi_{D}(\Gamma)$. Therefore, 
\[
H_{2n+1}(D\mathbb{Z}):=\ker(G\to H_{2n+1}(\mathbb{Z}/D))
\]
is a normal subgroup of finite index in $G$.

For ease of notation, we can identify $H_{2n+1}(\mathbb{R})$ with
$\mathbb{R}^{2n+1}$ via the map 
\begin{align}
\begin{pmatrix}1 & \mathbf{a} & c\\
 & I_{n} & \mathbf{b}\\
 &  & 1
\end{pmatrix} & \mapsto(\mathbf{a},\mathbf{b},c),\label{eq:Lie group id}
\end{align}
where $I_{n}$ is the identity matrix of size $n$, with group law
given by matrix multiplication: 
\begin{equation}
(\mathbf{a},\mathbf{b},c)\cdot(\mathbf{a}',\mathbf{b}',c')=(\mathbf{a}+\mathbf{a}',\mathbf{b}+\mathbf{b}',c+c'+\mathbf{a}\cdot\mathbf{b}').\label{eq:group law of H_2n+1}
\end{equation}
Similarly, we can identify the Lie algebra $\mathfrak{h}_{2n+1}(\mathbb{R})$
of $H_{2n+1}(\mathbb{R})$ with $\mathbb{R}^{2n+1}$ via the map 
\begin{align}
\begin{pmatrix}0 & \mathbf{a} & c\\
 & \mathbf{0}_{n} & \mathbf{b}\\
 &  & 0
\end{pmatrix} & \mapsto(\mathbf{a},\mathbf{b},c).\label{eq:Lie algebra id}
\end{align}
Note that depending on the context the same notation $(\mathbf{a},\mathbf{b},c)$
may denote different matrices. Let $\pi_{[1,2n]}:\mathbb{R}^{2n+1}\to\mathbb{R}^{2n}$\index{pi_{[1,2n]}:mathbb{R}^{2n+1}tomathbb{R}^{2n}, the projection 
of the first 2n coordinates@$\pi_{[1,2n]}:\mathbb{R}^{2n+1}\to\mathbb{R}^{2n}$, the projection
of the first $2n$ coordinates} or $\pi_{[1,2n]}:H_{2n+1}(\mathbb{R})\cong\mathbb{R}^{2n+1}\to\mathbb{R}^{2n}$\index{pi_{[1,2n]}:H_{2n+1}(mathbb{R})congmathbb{R}^{2n+1}tomathbb{R}^{2n}
, the projection of the first 2n coordinates@$\pi_{[1,2n]}:H_{2n+1}(\mathbb{R})\cong\mathbb{R}^{2n+1}\to\mathbb{R}^{2n}$,
the projection of the first $2n$ coordinates} be the projection of the first $2n$ coordinates and $\omega:\mathbb{R}^{2n}\times\mathbb{R}^{2n}\to\mathbb{R}$\index{omega:mathbb{R}^{2n}timesmathbb{R}^{2n}tomathbb{R}, the 
symplectic form@$\omega:\mathbb{R}^{2n}\times\mathbb{R}^{2n}\to\mathbb{R}$, the symplectic
form} be the symplectic form given by the nonsingular, skew-symmetric matrix
$\Omega=\begin{pmatrix} & I_{n}\\
-I_{n}
\end{pmatrix}$. Then, it is easy to verify the following formulas 
\begin{align}
(\mathbf{a},\mathbf{b},c)^{-1} & =(-\mathbf{a},-\mathbf{b},-c+\mathbf{a}\cdot\mathbf{b}),\label{eq:inverse}\\
(\mathbf{a},\mathbf{b},c)(\mathbf{a}',\mathbf{b}',c')(\mathbf{a},\mathbf{b},c)^{-1} & =(\mathbf{a}',\mathbf{b}',c'+\mathbf{a}\cdot\mathbf{b}'-\mathbf{a}'\cdot\mathbf{b}),\label{eq:conjugation}\\
(\mathbf{a},\mathbf{b},c)(\mathbf{a}',\mathbf{b}',c')(\mathbf{a},\mathbf{b},c)^{-1}(\mathbf{a}',\mathbf{b}',c')^{-1} & =(\mathbf{0},\mathbf{0},\mathbf{a}\cdot\mathbf{b}'-\mathbf{a}'\cdot\mathbf{b})=(\mathbf{0},\mathbf{0},\omega((\mathbf{a},\mathbf{b}),(\mathbf{a}',\mathbf{b}'))).\label{eq:commutator}
\end{align}
Notice that the center of $H_{2n+1}(\mathbb{R})$ is $Z=\{(\mathbf{0},\mathbf{0},c)\in H_{2n+1}(\mathbb{R})\mid c\in\mathbb{R}\}\cong\mathbb{R}$,
and that the matrix exponential $\exp:\mathfrak{h}_{2n+1}(\mathbb{R})\to H_{2n+1}(\mathbb{R})$
and the matrix logarithm $\log:H_{2n+1}(\mathbb{R})\to\mathfrak{h}_{2n+1}(\mathbb{R})$
are diffeomorphisms inverse to each other.

Let $g:\mathbb{N}_{0}\to G$ be a polynomial sequence. We can write\index{g(x)=(mathbf{a}(x),mathbf{b}(x),c(x)), a polynomial sequence in 
G=H_{2n+1}(mathbb{Z})@$g(x)=(\mathbf{a}(x),\mathbf{b}(x),c(x))$, a polynomial sequence
in $G=H_{2n+1}(\mathbb{Z})$}
\[
g(x)=(\mathbf{a}(x),\mathbf{b}(x),c(x))=(\mathbf{0},\mathbf{0},d(x))(\mathbf{a}(x),\mathbf{b}(x),\frac{1}{2}\mathbf{a}(x)\cdot\mathbf{b}(x)),
\]
where $\mathbf{a}=(g_{1,2},\ldots,g_{1,n+1})$\index{mathbf{a}=(g_{1,2},ldots,g_{1,n+1}), a vector of polynomial 
sequences in mathbb{Z}^{n}@$\mathbf{a}=(g_{1,2},\ldots,g_{1,n+1})$, a vector of polynomial sequences
in $\mathbb{Z}^{n}$} is a vector of polynomial sequences $g_{1,i}:\mathbb{N}_{0}\to\mathbb{Z}$,
$\mathbf{b}=(g_{2,n+2},\ldots,g_{n+1,n+2})$\index{mathbf{b}=(g_{2,n+2},ldots,g_{n+1,n+2}), a vector of polynomial 
sequences in mathbb{Z}^{n}@$\mathbf{b}=(g_{2,n+2},\ldots,g_{n+1,n+2})$, a vector of polynomial
sequences in $\mathbb{Z}^{n}$} is a vector of polynomial sequences $g_{j,n+2}:\mathbb{N}_{0}\to\mathbb{Z}$,
$c=g_{1,n+2}:\mathbb{N}_{0}\to\mathbb{Z}$\index{c=g_{1,n+2}, a polynomial sequence in mathbb{Z}@$c=g_{1,n+2}$, a polynomial sequence in $\mathbb{Z}$}
is a polynomial sequence, and $d=c-\frac{1}{2}\mathbf{a}\cdot\mathbf{b}:\mathbb{N}_{0}\to\mathbb{Z}$\index{d=c-frac{1}{2}mathbf{a}cdotmathbf{b}, a polynomial sequence 
in mathbb{Z}@$d=c-\frac{1}{2}\mathbf{a}\cdot\mathbf{b}$, a polynomial sequence
in $\mathbb{Z}$} is another polynomial sequence. We may replace $\mathbb{N}_{0}$
by $\mathbb{Z}_{\ge N}$ for some sufficiently large $N$ so that
$g_{i,j}(x)$ does not change its sign for all $x\ge N$. Since $(\mathbf{0},\mathbf{0},d(x))$
lies in the center of $G$, we have
\begin{align*}
\log g(x) & =\log(\mathbf{0},\mathbf{0},d(x))+\log(\mathbf{a}(x),\mathbf{b}(x),\frac{1}{2}\mathbf{a}(x)\cdot\mathbf{b}(x)),\\
 & =(\mathbf{0},\mathbf{0},d(x))+(\mathbf{a}(x),\mathbf{b}(x),0)=(\mathbf{a}(x),\mathbf{b}(x),d(x)).
\end{align*}
Consider the ordered product 
\[
\hat{g}:=\bigodot_{i=1}^{L}g:\mathbb{N}_{0}^{L}\to G;\ (x_{1},\ldots,x_{L})\mapsto g(x_{1})\cdots g(x_{L}).
\]
By the Baker-Campbell-Hausdorff formula, we have 
\begin{align*}
 & \hat{g}(x_{1},\ldots,x_{L})=\exp\left(\sum_{i=1}^{L}\log g(x_{i})\right)\exp\left(\dfrac{1}{2}\sum_{1\le i<j\le L}[\log g(x_{i}),\log g(x_{j})]\right)\\
= & \left(\sum_{i=1}^{L}\mathbf{a}(x_{i}),\sum_{i=1}^{L}\mathbf{b}(x_{i}),\sum_{i=1}^{L}d(x_{i})+\dfrac{1}{2}\left(\sum_{i=1}^{L}\mathbf{a}(x_{i})\right)\cdot\left(\sum_{i=1}^{L}\mathbf{b}(x_{i})\right)\right)\\
 & \left(\mathbf{0},\mathbf{0},\dfrac{1}{2}\sum_{1\le i<j\le L}\omega((\mathbf{a}(x_{i}),\mathbf{b}(x_{i})),(\mathbf{a}(x_{j}),\mathbf{b}(x_{j})))\right).
\end{align*}
Then, by the proof of Proposition \ref{prop:symmetrization}, the
product\index{tilde{g}(x_{1},ldots,x_{L})=g(x_{1})cdots g(x_{L})g(x_{L})cdots g(x_{1})
, symmetric polynomial map in L variables@$\tilde{g}(x_{1},\ldots,x_{L})=g(x_{1})\cdots g(x_{L})g(x_{L})\cdots g(x_{1})$,
symmetric polynomial map in $L$ variables}
\[
\tilde{g}(x_{1},\ldots,x_{L})=g(x_{1})\cdots g(x_{L})g(x_{L})\cdots g(x_{1})
\]
is a symmetric polynomial map in $L$ variables $x_{1},x_{2},\ldots,x_{L}$
of the form
\[
\exp\left(2\sum_{i=1}^{L}\log g(x_{i})\right)=\left(\sum_{i=1}^{L}2\mathbf{a}(x_{i}),\sum_{i=1}^{L}2\mathbf{b}(x_{i}),\sum_{i=1}^{L}2d(x_{i})+2\left(\sum_{i=1}^{L}\mathbf{a}(x_{i})\right)\cdot\left(\sum_{i=1}^{L}\mathbf{b}(x_{i})\right)\right).
\]
Let $B$ \index{Binmathbb{N}, the least upper bound of the total degrees@$B\in\mathbb{N}$, the least upper bound of the total degrees}
be the least upper bound as in the simple observation and $a_{i,j}^{k}$
be the coefficient of the term $x^{k}$ in the polynomial $g_{i,j}(x)$\index{$a_{i,j}^{k}$, the coefficient of the term $x^{k}$ in the polynomial
$g_{i,j}(x)$}. Then, we can write 
\[
g_{i,j}(x)=\sum_{k=1}^{B}a_{i,j}^{k}x^{k}=(a_{i,j}^{0},a_{i,j}^{1},\ldots,a_{i,j}^{B})\cdot(1,x,\ldots,x^{B}).
\]
To ease the notation, we write $\mathbf{a}_{i,j}^{B}=(a_{i,j}^{0},a_{i,j}^{1},\ldots,a_{i,j}^{B})$
and $\mathbf{s}_{B}=(s_{0},s_{1},\ldots,s_{B})$, where $s_{j}=\sum_{i=1}^{L}x_{i}^{j}$
for $j=0,\ldots,B$. Each $\tilde{g}_{i,j}$ is a symmetric polynomial
of total degree $\le B$ and thus can be written as a polynomial expression
with rational coefficients in the power sum symmetric polynomials
$s_{1},s_{2},\ldots,s_{B}$, i.e., \index{tilde{g}_{i,j}, the (i,j)-entry of the polynomial sequence tilde{g}
 in a matrix form@$\tilde{g}_{i,j}$, the $(i,j)$-entry of the polynomial sequence
$\tilde{g}$ in a matrix form}
\[
\tilde{g}_{i,j}(x_{1},\cdots,x_{L})=p_{i,j}(s_{1},\ldots,s_{B}),
\]
where $p_{i,j}$\index{$p_{i,j}$, an integer-valued polynomial in $B$ variables}
is an integer-valued polynomial in $B$ variables $s_{1},\ldots,s_{B}\in\mathbb{N}_{0}$
with rational coefficients of the following forms:  
\[
p_{1,i}(s_{1},\ldots,s_{B})=2\mathbf{a}_{1,i}^{B}\cdot\mathbf{s}_{B},\qquad p_{i,n+2}(s_{1},\ldots,s_{B})=2\mathbf{a}_{i,n+2}^{B}\cdot\mathbf{s}_{B},
\]
\begin{align*}
p_{1,n+2}(s_{1},\ldots,s_{B}) & =2\mathbf{a}_{1,n+2}^{B}\cdot\mathbf{s}_{B}-\sum_{k=2}^{n+1}\sum_{j=0}^{B}\sum_{l=0}^{j}a_{1,k}^{l}a_{k,n+2}^{j-l}s_{j}+2\sum_{k=2}^{n+1}\left(\mathbf{a}_{1,k}^{B}\cdot\mathbf{s}_{B}\right)\left(\mathbf{a}_{k,n+2}^{B}\cdot\mathbf{s}_{B}\right).
\end{align*}
Therefore, we can define the following continuous polynomial map in
$B$ variables $s_{1},\ldots,s_{B}\in\mathbb{R}_{\ge0}$: \index{p:=(p_{1,2},ldots,p_{1,n+1},p_{2,n+2},ldots,p_{n+1,n+2},p_{1,n+2})
, a continuous polynomial map in B variables@$p:=(p_{1,2},\ldots,p_{1,n+1},p_{2,n+2},\ldots,p_{n+1,n+2},p_{1,n+2})$,
a continuous polynomial map in $B$ variables}
\[
p:=(p_{1,2},\ldots,p_{1,n+1},p_{2,n+2},\ldots,p_{n+1,n+2},p_{1,n+2}):\mathbb{R}_{\ge0}^{B}\to H_{2n+1}(\mathbb{R}).
\]
Then, $\log p(s_{1},\ldots,s_{B})$ is given by 
\[
\left(2\mathbf{a}_{1,2}^{B}\cdot\mathbf{s}_{B},\ldots,2\mathbf{a}_{1,n+1}^{B}\cdot\mathbf{s}_{B},2\mathbf{a}_{2,n+2}^{B}\cdot\mathbf{s}_{B},\ldots,2\mathbf{a}_{n+1,n+2}^{B}\cdot\mathbf{s}_{B},2\mathbf{a}_{1,n+2}^{B}\cdot\mathbf{s}_{B}-\sum_{k=2}^{n+1}\sum_{j=0}^{B}\sum_{l=0}^{j}a_{1,k}^{l}a_{k,n+2}^{j-l}s_{j}\right),
\]
whose Jacobian matrix with respect to the $s_{j}$'s ($j=1,\ldots,B$)
\index{mathbf{J}, Jacobian matrix of log p(s_{1},ldots,s_{B}) with 
respect to the s_{j}'s@$\mathbf{J}$, Jacobian matrix of $\log p(s_{1},\ldots,s_{B})$ with
respect to the $s_{j}$'s} is given by 
\[
\mathbf{J}=\begin{pmatrix}2a_{1,2}^{1} & \cdots & 2a_{1,n+1}^{1} & 2a_{2,n+2}^{1} & \cdots & 2a_{n+1,n+2}^{1} & 2a_{1,n+2}^{1}-\sum_{k=2}^{n+1}\sum_{l=0}^{1}a_{1,k}^{l}a_{k,n+2}^{1-l}\\
\vdots & \ddots & \vdots & \vdots & \ddots & \vdots & \vdots\\
2a_{1,2}^{B} & \cdots & 2a_{1,n+1}^{B} & 2a_{2,n+2}^{B} & \cdots & 2a_{n+1,n+2}^{B} & 2a_{1,n+2}^{B}-\sum_{k=2}^{n+1}\sum_{l=0}^{B}a_{1,k}^{l}a_{k,n+2}^{B-l}
\end{pmatrix}.
\]
Let $\mathbf{J}_{0}$ \index{mathbf{J}_{0}, the matrix formed by the first 2n columns of mathbf{J}@$\mathbf{J}_{0}$, the matrix formed by the first $2n$ columns of
$\mathbf{J}$} be the matrix formed by the first $2n$ columns of $\mathbf{J}$
and let $\mathbf{J}_{1}$ \index{mathbf{J}_{1}, the matrix which contains the constant vector(2a_{1,2}^{0},cdots,2a_{1,n+1}^{0},2a_{2,n+2}^{0},cdots,2a_{n+1,n+2}^{0},2a_{1,n+2}^{0}-sum_{k=2}^{n+1}a_{1,k}^{0}a_{k,n+2}^{0})
as its first row and mathbf{J} as the remaining rows@$\mathbf{J}_{1}$, the matrix which contains the constant vector
\[
(2a_{1,2}^{0},\cdots,2a_{1,n+1}^{0},2a_{2,n+2}^{0},\cdots,2a_{n+1,n+2}^{0},2a_{1,n+2}^{0}-\sum_{k=2}^{n+1}a_{1,k}^{0}a_{k,n+2}^{0})
\]
as its first row and $\mathbf{J}$ as the remaining rows} be the matrix which has the constant vector
\begin{equation}
(2a_{1,2}^{0},\cdots,2a_{1,n+1}^{0},2a_{2,n+2}^{0},\cdots,2a_{n+1,n+2}^{0},2a_{1,n+2}^{0}-\sum_{k=2}^{n+1}a_{1,k}^{0}a_{k,n+2}^{0})\label{eq:1st row of J_1}
\end{equation}
as its first row and $\mathbf{J}$ as the remaining rows. 

Suppose that $g\mod N$ is non-constant for any normal subgroup $N$
of infinite index in $G$.  Since any infinite index normal subgroup
of $G/[G,G]$ corresponds to an infinite index normal subgroup of
$G$ containing $[G,G]$, we see that the induced map 
\[
g\mod[G,G]:\mathbb{N}_{0}\to\mathbb{Z}^{2n};\ x\mapsto(\mathbf{a}(x),\mathbf{b}(x)),
\]
satisfies the hypotheses of Theorem \ref{thm:WP for Z^m}. By the
proof of Theorem \ref{thm:WP for Z^m}, $g\mod[G,G]$ has degree at
least $2n$ and $\mathbf{J}_{0}$ has rank $2n$ and $p_{1,2},\ldots,p_{1,n+1},p_{2,n+2},\ldots,p_{n+1,n+2}$
are linearly independent in variables $s_{1},\ldots,s_{B}$ with $B\ge2n$.
In this case, the rank of $\mathbf{J}$ is $2n$, if and only if the
last column of $\mathbf{J}$ is a $\mathbb{Q}$-linear combination
of the first $2n$ columns of $\mathbf{J}$, if and only if $d(x)$\index{d(x)=mathbf{u}cdotmathbf{a}(x)+mathbf{v}cdotmathbf{b}(x)+w, 
the general degenerate case@$d(x)=\mathbf{u}\cdot\mathbf{a}(x)+\mathbf{v}\cdot\mathbf{b}(x)+w$,
the general degenerate case} has the form 
\begin{equation}
d(x)=\mathbf{u}\cdot\mathbf{a}(x)+\mathbf{v}\cdot\mathbf{b}(x)+w,\label{eq:d(x)}
\end{equation}
where $\mathbf{u}=(u_{1},\ldots,u_{n})\in\mathbb{Q}^{n}$, $\mathbf{v}=(v_{1},\ldots,v_{n})\in\mathbb{Q}^{n}$,
and $w\in\mathbb{Q}$.  

Just like the abelian case, we consider the following matrix associated
to $\log g$: 
\begin{align*}
\mathbf{J}(\log g)(x) & =\begin{pmatrix}g_{1,2}^{(1)} & \cdots & g_{1,n+1}^{(1)} & g_{2,n+2}^{(1)} & \cdots & g_{n+1,n+2}^{(1)} & d^{(1)}\\
\vdots & \ddots & \vdots & \vdots & \ddots & \vdots & \vdots\\
g_{1,2}^{(B)} & \cdots & g_{1,n+1}^{(B)} & g_{2,n+2}^{(B)} & \cdots & g_{n+1,n+2}^{(B)} & d^{(B)}
\end{pmatrix}(x).
\end{align*}
Notice that entries of $\mathbf{J}(\log g)(0)$ are nothing but positive
scalar multiples of coefficients of $\log g(x)$ or the corresponding
entries of the above matrix $\mathbf{J}$. Similarly, let $\mathbf{J}_{0}(\log g)(x)$
be the matrix formed by the first $2n$ columns of $\mathbf{J}(\log g)(x)$
and let $\mathbf{J}_{1}(\log g)(x)$ be the matrix which contains
\begin{equation}
(g_{1,2}(x),\cdots,g_{1,n+1}(x),g_{2,n+2}(x),\cdots g_{n+1,n+2}(x),d(x))\label{eq:1st row of J_1(log g)}
\end{equation}
as its first row and $\mathbf{J}(g)$ as the remaining rows. 

For any pair $(a,b)\in\mathbb{N}\times\mathbb{N}_{0}$, we can also
consider the polynomial sequence $g(ax+b)$, i.e., the composition
of $g(x)$ with the affine translation $ax+b$. By the chain rule
of derivatives, the matrix $\mathbf{J}(\log g(ax+b))(x)$ is given
by 
\[
\begin{pmatrix}ag_{1,2}^{(1)} & \cdots & ag_{1,n+1}^{(1)} & ag_{2,n+2}^{(1)} & \cdots & ag_{n+1,n+2}^{(1)} & ad^{(1)}\\
\vdots & \ddots & \vdots & \vdots & \ddots & \vdots & \vdots\\
a^{B}g_{1,2}^{(B)} & \cdots & a^{B}g_{1,n+1}^{(B)} & a^{B}g_{2,n+2}^{(B)} & \cdots & a^{B}g_{n+1,n+2}^{(B)} & a^{B}d^{(B)}
\end{pmatrix}(ax+b).
\]
If $g_{i}(x)=(\mathbf{a}_{i}(x),\mathbf{b}_{i}(x),c_{i}(x)):\mathbb{N}_{0}\to G$,
$1\le i\le m$, are polynomial sequences, then the logarithm of their
ordered product is given by 
\begin{align*}
\log(g_{1}g_{2}\cdots g_{m}) & =\sum_{i=1}^{m}\log g_{i}+\frac{1}{2}\sum_{1\le i<j\le m}[\log g_{i},\log g_{j}]\\
 & =\sum_{i=1}^{m}(\mathbf{a}_{i},\mathbf{b}_{i},d_{i})+\frac{1}{2}\sum_{1\le i<j\le m}(\mathbf{0},\mathbf{0},\mathbf{a}_{i}\cdot\mathbf{b}_{j}-\mathbf{a}_{j}\cdot\mathbf{b}_{i})
\end{align*}
and $\mathbf{J}(g_{1}g_{2}\cdots g_{m})(x)=\sum_{i=1}^{m}\mathbf{J}(\log g_{i})(x)+\frac{1}{2}\sum_{1\le i<j\le m}\mathbf{J}([\log g_{i},\log g_{j}])(x)$,
where 
\[
\mathbf{J}([\log g_{i},\log g_{j}])(x)=\begin{pmatrix}\mathbf{0} & (\mathbf{a}_{i}\cdot\mathbf{b}_{j}^{(1)}-\mathbf{a}_{j}\cdot\mathbf{b}_{i}^{(1)}+\mathbf{a}_{i}^{(1)}\cdot\mathbf{b}_{j}-\mathbf{a}_{j}^{(1)}\cdot\mathbf{b}_{i})(x)\\
\vdots & \vdots\\
\mathbf{0} & \sum_{k=0}^{B}\binom{B}{k}(\mathbf{a}_{i}^{(k)}\cdot\mathbf{b}_{j}^{(B-k)}-\mathbf{a}_{j}^{(k)}\cdot\mathbf{b}_{i}^{(B-k)})(x)
\end{pmatrix}.
\]

We need the following lemma to handle the degenerate case. 
\begin{lem}
\label{lem:lem4deg} If $g\mod N:\mathbb{N}_{0}\to G\twoheadrightarrow G/N$
is non-constant for any normal subgroup $N$ of infinite index in
$G$, then there exists a finite number of pairs $(a_{i},b_{i})\in\mathbb{N}\times\mathbb{N}_{0}$,
$1\le i\le m$, such that the associated matrix $\mathbf{J}(\log h)(0)$
of the polynomial sequence 
\[
h(x):=g(a_{1}x+b_{1})\cdots g(a_{m}x+b_{m}):\mathbb{N}_{0}\to G
\]
has rank $2n+1$. 
\end{lem}
\begin{proof}
Since $g\mod[G,G]=(\mathbf{a},\mathbf{b}):\mathbb{N}_{0}\to G/[G,G]\cong\mathbb{Z}^{2n}$
satisfies the hypotheses of Theorem \ref{thm:WP for Z^m}, by the
remark after the proof of that theorem, we can conclude $\mathbf{J}_{0}(\log g)(b)=\mathbf{J}_{0}((\mathbf{a},\mathbf{b}))(b)$
has rank $2n$ for any $b\in\mathbb{N}_{0}$. If $\rank\mathbf{J}(\log g)(b)=2n+1$
for some $b\in\mathbb{N}_{0}$, then we are done. Suppose $\rank\mathbf{J}(\log g)(b)=\rank\mathbf{J}_{0}(\log g)(b)=2n$
for all $b\in\mathbb{N}_{0}$, i.e., $d(x)$ is of the form (\ref{eq:d(x)}). 

For any tuple $(a_{1},\ldots,a_{m},b)\in\mathbb{N}^{m}\times\mathbb{N}_{0}$,
consider the ordered product 
\[
h(x)=g(a_{1}x+b)g(a_{2}x+b)\cdots g(a_{m}x+b).
\]
Then, we have 
\begin{align*}
\mathbf{J}(\log h)(0) & =\sum_{i=1}^{m}\mathbf{J}(\log g(a_{i}x+b))(0)+\frac{1}{2}\sum_{1\le i<j\le m}\mathbf{J}([\log g(a_{i}x+b),\log g(a_{j}x+b)])(0)\\
 & =\diag\left(\sum_{i=1}^{m}a_{i},\sum_{i=1}^{m}a_{i}^{2},\ldots,\sum_{i=1}^{m}a_{i}^{B}\right)\mathbf{J}(\log g)(b)\\
 & +\frac{1}{2}\sum_{1\le i<j\le m}\sum_{l=2}^{n+1}\begin{pmatrix}\mathbf{0} & \sum_{k=0}^{1}\binom{1}{k}(a_{i}^{k}a_{j}^{1-k}-a_{i}^{1-k}a_{j}^{k})g_{1,l}^{(k)}g_{l,n+2}^{(1-k)}(b)\\
\vdots & \vdots\\
\mathbf{0} & \sum_{k=0}^{B}\binom{B}{k}(a_{i}^{k}a_{j}^{B-k}-a_{i}^{B-k}a_{j}^{k})g_{1,l}^{(k)}g_{l,n+2}^{(B-k)}(b)
\end{pmatrix}.
\end{align*}

Now we fix $b$ and let $S_{m}(b)$ be the subset of $\mathbb{N}^{m}$
consisting of tuples $(a_{1},\ldots,a_{m})$ such that the rank of
$\mathbf{J}(\log h)(0)$ is $2n$. Since $\rank\mathbf{J}(\log g)(b)=2n$,
so is $\rank\sum_{i=1}^{m}\mathbf{J}(\log g(a_{i}x+b))(0)$. Hence,
the rank of $\mathbf{J}(\log h)(0)$ is $2n$ if and only if the last
column of 
\[
\frac{1}{2}\sum_{1\le i<j\le m}\mathbf{J}([\log g(a_{i}x+b),\log g(a_{j}x+b)])(0)
\]
is a linear combination of the first $2n$ columns of $\sum_{i=1}^{m}\mathbf{J}(\log g(a_{i}x+b))(0)$.
Since $g_{i,j}(x)$ are polynomials in $\mathbb{Q}[x]$, we obtain
a system of $B$ polynomial equations in $\mathbb{Q}[a_{1},\ldots,a_{m}]$
of degree at most $B$. So $S_{m}(b)$ consists of tuples $(a_{1},\ldots,a_{m})\in\mathbb{N}^{m}$
that are solutions to these polynomial equations. For $m>1$, we need
to analyze when $S_{m}(b)=\mathbb{N}^{m}$, i.e., when these $B$
polynomials vanish on all of $\mathbb{N}^{m}$. However, comparing
similar terms, we see that this can happen if and only if the last
column of $\frac{1}{2}\sum_{1\le i<j\le m}\mathbf{J}([\log g(a_{i}x+b),\log g(a_{j}x+b)])(0)$
is a vector of zero polynomials in variables $a_{1},\ldots,a_{m}$.
Therefore, as the coefficient of $a_{i}^{k}a_{j}^{L-k}$, 
\[
\frac{1}{2}\sum_{l=2}^{n+1}\binom{L}{k}\left(g_{1,l}^{(k)}g_{l,n+2}^{(L-k)}(b)-g_{1,l}^{(L-k)}g_{l,n+2}^{(k)}(b)\right)
\]
must vanish for all $L=1,\ldots,B$ and $k=0,\ldots,L$. Since this
works for all $b\in\mathbb{N}_{0}$, we see that
\begin{align*}
(\mathbf{a}\cdot\mathbf{b})^{(L)}(x) & =\sum_{l=2}^{n+1}\sum_{k=0}^{L}\binom{L}{k}g_{1,l}^{(k)}g_{l,n+2}^{(L-k)}(x)
\end{align*}
vanishes for any odd $L=1,\ldots,B$, and thus $\mathbf{a}\cdot\mathbf{b}$
must be a constant $C$. However, if $\mathbf{a}\cdot\mathbf{b}=C$,
then $(\mathbf{a},\mathbf{b})$ cannot satisfy the hypotheses of \cite[Thm 2]{Hu2021}.
So for sufficiently large $m$, we must have $S_{m}(b)\ne\mathbb{N}^{m}$
for some $b\in\mathbb{N}_{0}$ and thus the proof is complete. 
\end{proof}
 In the proof of the main result, we will need the following maps
of sets \index{delta:H_{2n+1}(mathbb{R})to H_{2n+1}(mathbb{R});(mathbf{a},mathbf{b},c)mapsto(mathbf{a},mathbf{b},c-frac{1}{2}mathbf{a}cdotmathbf{b})@$\delta:H_{2n+1}(\mathbb{R})\to H_{2n+1}(\mathbb{R});(\mathbf{a},\mathbf{b},c)\mapsto(\mathbf{a},\mathbf{b},c-\frac{1}{2}\mathbf{a}\cdot\mathbf{b})$}\index{iota:H_{2n+1}(mathbb{R})to H_{2n+1}(mathbb{R});(mathbf{a},mathbf{b},c)mapsto(mathbf{a},mathbf{b},c+frac{1}{2}mathbf{a}cdotmathbf{b})@$\iota:H_{2n+1}(\mathbb{R})\to H_{2n+1}(\mathbb{R});(\mathbf{a},\mathbf{b},c)\mapsto(\mathbf{a},\mathbf{b},c+\frac{1}{2}\mathbf{a}\cdot\mathbf{b})$}
\begin{align*}
\delta:H_{2n+1}(\mathbb{R}) & \to H_{2n+1}(\mathbb{R}), & \iota:H_{2n+1}(\mathbb{R}) & \to H_{2n+1}(\mathbb{R}),\\
(\mathbf{a},\mathbf{b},c) & \mapsto(\mathbf{a},\mathbf{b},c-\frac{1}{2}\mathbf{a}\cdot\mathbf{b}), & (\mathbf{a},\mathbf{b},c) & \mapsto(\mathbf{a},\mathbf{b},c+\frac{1}{2}\mathbf{a}\cdot\mathbf{b}),
\end{align*}
which are bijective and mutually inverses to each other. 
\begin{proof}[Proof of Theorem \ref{thm:WP for Heisenberg group}]
 Since $g\mod N$ is non-constant for any normal subgroup $N$ of
infinite index in $G$, $\langle g\rangle$ must have finite index
in $G$.  Set $B$ as before and let $L'\in\mathbb{N}$\index{L'inmathbb{N}@$L'\in\mathbb{N}$}
be the least number such that
\[
\max\{\hat{d}_{i,j}(L)\mid1\le i<j\le n,1\le L\le L'\}=B.
\]
By Kamke's key theorem, there exist positive integers $A$\index{Ainmathbb{N}, in Kamke's key theorem@$A\in\mathbb{N}$, in Kamke's key theorem},
$L''$\index{L''inmathbb{N}@$L''\in\mathbb{N}$}, and positive numbers
$i_{1}$\index{i_{1}inmathbb{R}_{>0}@$i_{1}\in\mathbb{R}_{>0}$} and
$i_{\nu}$, $J_{\nu}$ \index{i_{nu} and J_{nu}, positive numbers with 0<i_{nu}<J_{nu}, nu=2,3,ldots,B@$i_{\nu}$ and $J_{\nu}$, positive numbers with $0<i_{\nu}<J_{\nu}$,
$\nu=2,3,\ldots,B$} with $0<i_{\nu}<J_{\nu}$, $\nu=2,3,\ldots,B$, such that for each
$B$ integers $s_{1},\ldots,s_{B}$, divisible by $A$ and subject
to the following conditions 
\[
s_{1}\in(i_{1},\infty);\ s_{\nu}/s_{1}^{\nu}\in(i_{\nu},J_{\nu})\text{ for }\nu=2,3,\ldots,B,
\]
the $B$ equations $s_{\nu}=\sum_{\kappa=1}^{L''}x_{\kappa}^{\nu}$,
$\nu=1,2,\ldots,B$, are simultaneously solvable by integers $x_{\kappa}\ge0$.
Let $L$ \index{Lin Amathbb{N}, Lgemax{L',L''}@$L\in A\mathbb{N}$, $L\ge\max\{L',L''\}$}
be the least integer $\ge\max\{L',L''\}$ and divisible by $A$. 

By \cite[Cor 8]{Hu2020}, the subgroup $\langle\hat{g}\rangle=\langle\bigodot_{i=1}^{L}g\rangle$
has finite index in $\langle g\rangle$. Since $\langle g\rangle$
is finitely generated and nilpotent, by a result due to Mal'tsev (cf.
\cite[Thm 2.23]{CMZ2017}), $\langle\hat{g}\restriction_{\mathbb{N}_{0}}\rangle=\langle g^{L}\rangle$\footnote{Here, $\restriction$ is the standard restriction symbol in \LaTeX.}
has finite index in $\langle g\rangle$ and thus in $\langle\hat{g}\rangle$.
By \cite[Thm 15]{Hu2020}, $\langle\tilde{g}\rangle$ has finite index
in $\langle\hat{g}\rangle$ and thus also in $\langle g\rangle$.
Then, we have $\langle\tilde{g}(\mathbb{N}_{0}^{L})\rangle\subseteq\langle p(\mathbb{N}_{0}^{B})\rangle\subseteq\langle p(\mathbb{R}_{\ge0}^{B})\rangle\subseteq H_{2n+1}(\mathbb{R})$.
Since $\langle\tilde{g}(\mathbb{N}_{0}^{L})\rangle$ has finite index
in $G$, $\langle p(\mathbb{R}_{\ge0}^{B})\rangle$ contains a neighborhood
of the identity of $H_{2n+1}(\mathbb{R})$ and thus $\langle p(\mathbb{R}_{\ge0}^{B})\rangle=H_{2n+1}(\mathbb{R})$.
 In particular, $\langle p(U)\rangle=H_{2n+1}(\mathbb{R})$ for any
subset $U$ of $\mathbb{R}_{\ge0}^{B}$ with nonempty interior. 

Case 1: If $d(x)$ is not of the form (\ref{eq:d(x)}), then the rank
of $\mathbf{J}$ is $2n+1$. Let $U(B,L)$ \index{$U(B,L)$, the Kamke domain}
be the Kamke domain, which by \cite[Lem 8]{Hu2020} contains a proper
polynomial set given by some continuous polynomial map $q$\index{$q$, a continuous polynomial map whose image is a proper polynomial
set insight the Kamke domain}, and set $f=p\circ q$\index{f=pcirc q@$f=p\circ q$}. Therefore,
the dimension of $p(\mathbb{R}_{\ge0}^{B})$ or the proper polynomial
set $U=\Ima f$\index{U=Ima f@$U=\Ima f$} is $2n+1$ as desired.
Moreover, we have 
\[
\tilde{g}(\mathbb{N}_{0}^{2L})\supset p(U(B,L)\cap A\mathbb{N}_{0}^{B})=p(U(B,L)\cap A\mathbb{Z}^{B})\supset U\cap p(A\mathbb{Z}^{B}).
\]
 Then, $\delta\circ p$ is affine linear in $s_{1},\ldots,s_{B}$.
It is easy to see that the Jacobian matrix of $\delta\circ p$ with
respect to $s_{1},\ldots,s_{B}$ is the same as the one of $\log p$.
If necessary, we can replace $A$ by $kA$ for some $k\in\mathbb{N}$
so that $p(A\mathbb{Z}^{B})\in G$. Since the first row vector (\ref{eq:1st row of J_1})
of $\mathbf{J}_{1}$ is a $\mathbb{Q}$-linear combination of the
row vectors of $\mathbf{J}$, we can replace $L$ by $k'L$ for some
$k'\in\mathbb{N}$ so that $\delta\circ p(A\mathbb{Z}^{B})$ is a
lattice with integral entries in the underlying Euclidean space of
$H_{2n+1}(\mathbb{R})$ and the subgroup $H=H_{2n+1}(D\mathbb{Z})$\index{H=H_{2n+1}(Dmathbb{Z})@$H=H_{2n+1}(D\mathbb{Z})$}
lies in $\delta\circ p(A\mathbb{Z}^{B})$ for some $D\in2\mathbb{N}$\index{Din2mathbb{N}@$D\in2\mathbb{N}$}.
Thus, $\iota(H)$\index{iota(H), the image of the map iota restricted on H@$\iota(H)$, the image of the map $\iota$ restricted on $H$}
lies in $p(A\mathbb{Z}^{B})$. But $\iota(H)=H$ as sets, since $D$
is even and divides $D\cdot D/2$. Then, $H$ is a finite index subgroup
of $G$ such that $H\subset p(A\mathbb{Z}^{B})$. Let $\phi:H\hookrightarrow H_{2n+1}(\mathbb{R})$
\index{phi:Hhookrightarrow H_{2n+1}(mathbb{R}), the inclusion map@$\phi:H\hookrightarrow H_{2n+1}(\mathbb{R})$, the inclusion map}
be the inclusion map. Then, $V=\phi^{-1}(U)=U\cap H\subset\tilde{g}(\mathbb{N}_{0}^{2L})$
\index{V=phi^{-1}(U)=Ucap H@$V=\phi^{-1}(U)=U\cap H$} is a proper
polynomial subset in $H$ and each element of $V$ can be written
as a product of exactly $2L$ elements in the sequence $g_{0},g_{1},g_{2},\ldots$.

Case 2: If $d(x)$ is of the form (\ref{eq:d(x)}), then the rank
of $\mathbf{J}$ is $2n$. Then, by Lemma \ref{lem:lem4deg}, we can
find a finite number of pairs $(a_{i},b_{i})\in\mathbb{N}\times\mathbb{N}_{0}$,
$1\le i\le m$, such that the associated matrix $\mathbf{J}(\log h)(0)$
of the polynomial sequence $h(x):=g(a_{1}x+b_{1})\cdots g(a_{m}x+b_{m}):\mathbb{N}_{0}\to G$
has rank $2n+1$. Working with $h(x)$ instead of $g(x)$, we are
reduced to the previous case. 
\end{proof}

\section{Waring's Problem for Locally Nilpotent Groups \label{sec:WP for LNGs}}

The basic setting of Waring's problem for locally nilpotent groups
is as follows: Let $G$ be a locally nilpotent group, $g:\mathbb{N}_{0}\to G$
be an arbitrary polynomial sequence of degree $d$, and $[g]$ (resp.
$\langle g\rangle$) be the semigroup (resp. locally nilpotent subgroup)
generated by the polynomial sequence $g_{0},g_{1},g_{2},\ldots$.
By \cite[Prop 2]{Hu2020}, we may assume that $G=\langle g\rangle$
is finitely generated and nilpotent.

The whole discussion of Waring's problem for finitely generated nilpotent
groups is divided into different cases, according to the degree of
$g$, and the cardinality and nilpotency class of the group $\langle g\rangle$.
But we only state and prove some nontrivial cases here.

We may assume that $\langle g\rangle$ is infinite. Otherwise, one
can easily show that $[g]=\langle g\rangle$ is an open polynomial
set such that each element in $[g]$ can be written as a product of
at most $M$ elements in the sequence $g_{0},g_{1},g_{2},\ldots$
for some $M\in\mathbb{N}$. 

We may as well assume that $d\ge1$. If $d\le0$, i.e., $g$ is constant,
then it is easy to prove that there exists $M\in\mathbb{N}$ and a
proper polynomial set $V$ in $\langle g\rangle$ such that each element
in $V\cap[g]$ can be written as a product of at most $M$ elements
in the sequence $g_{0},g_{1},g_{2},\ldots$. 

\subsection{Case in which $d=1$}

The following example sheds some light on the proof of next theorem. 
\begin{example*}
Consider the quotient group $G$ of $H_{3}(\mathbb{Z})=<x,y\mid[x,[x,y]]=[y,[x,y]]=1>$
given by 
\[
G=<x,y\mid[x,[x,y]]=[y,[x,y]]=1=x^{n}y^{-m}>,\text{ for some }n\ne0.
\]
Then, $G^{\text{ab}}=<\bar{x},\bar{y}\mid n\bar{x}-m\bar{y}=0>$ has
rank $1$. Consider the polynomial sequence $g:\mathbb{N}_{0}\to G$;
$i\mapsto xy^{i}$ of degree $1$. Then, $g_{0}=x$, $g_{1}=xy$,
and $g_{0}^{-1}g_{1}=y$. Hence, $\langle g\rangle=G$ but $\langle g\rangle^{\text{ab}}$
has rank $1$. 
\end{example*}
\begin{thm}
\label{thm:WP for deg=00003D1} Let $g:\mathbb{N}_{0}\to G$ be a
polynomial sequence of degree $d=1$ such that $\langle g\rangle$
is infinite. If $g_{0}\ne1_{G}$ and $g\mod N$ is non-constant for
any normal subgroup $N$ of infinite index in $\langle g\rangle$,
then there exists $M\in\mathbb{N}$, a subgroup $H$ of finite index
in $\langle g\rangle$, and a proper polynomial set $V$ in $H$ such
that each element in $V\cap[g]$ can be written as a product of at
most $M$ elements in the sequence $g_{0},g_{1},g_{2},\ldots$. 
\end{thm}
\begin{proof}
Since $\mathbb{N}_{0}$ is a commutative monoid, by \cite[Prop 1]{Hu2020},
$g$ is affine multiplicative, and $l:=g_{0}^{-1}g$ and $r:=gg_{0}^{-1}$
are multiplicative, i.e., $l_{i}=l_{1}^{i}$ and $r_{i}=r_{1}^{i}$
for all $i\in\mathbb{N}_{0}$. Then, 
\[
\langle g\rangle=\langle g_{0},g_{1},g_{2},\ldots\rangle=\langle g_{0},l_{1}\rangle=\langle g_{0},r_{1}\rangle=\langle g_{0},g_{1}\rangle
\]
 and thus $\langle g\rangle^{\text{ab}}=\langle g\rangle/[\langle g\rangle,\langle g\rangle]=\langle\bar{g}_{0},\bar{l}_{1}\rangle$
has rank $0$, $1$ or $2$. 

If $\langle g\rangle^{\text{ab}}$ has rank $0$, then it is finite.
 Then, $\langle g\rangle$ is also finite, which contradicts our
assumption. Suppose that $\langle g\rangle^{\text{ab}}=\langle\bar{g}_{0},\bar{l}_{1}\rangle$
has rank $2$. Then, we have $\langle g\rangle^{\text{ab}}=\langle\bar{g}_{0},\bar{l}_{1}\rangle$
and $\langle\bar{l}_{1}\rangle$ is a subgroup of infinite index in
$\langle g\rangle^{\text{ab}}$. Let $N$ be the normal subgroup of
$\langle g\rangle$ such that $\langle g\rangle/N\cong\langle g\rangle^{\text{ab}}/\langle\bar{l}_{1}\rangle$.
Since $\langle g\rangle^{\text{ab}}/\langle\bar{l}_{1}\rangle\cong\mathbb{Z}$,
$N$ has infinite index in $\langle g\rangle$. Since $g_{i}=g_{0}l_{i}=g_{0}l_{1}^{i}\in g_{0}N$,
it follows that $g\mod N$ is a constant, which also contradicts our
assumption. 

So $\langle g\rangle^{\text{ab}}=\langle\bar{g}_{0},\bar{l}_{1}\rangle$
must have rank $1$. If $\langle l_{1}\rangle$ is a subgroup of infinite
index in $\langle g\rangle$, then by \cite[Lem 6]{Hu2020} there
exists a normal subgroup $N$ of infinite index in $\langle g\rangle$
containing $\langle l_{1}\rangle$. Since $g_{i}=g_{0}l_{i}=g_{0}l_{1}^{i}\in g_{0}N$,
it follows that $g\mod N$ is a constant. Hence, $\langle l_{1}\rangle$
must be a subgroup of finite index in $\langle g\rangle$ and thus
$l_{1}$ has infinite order.  Let $H=\langle l_{1}\rangle\cong\mathbb{Z}$
be a subgroup of $\langle g\rangle$. 

If $g_{0}$ has finite order, say $n$, then we consider $h:=g_{0}^{n-1}g$.
Then, $h_{i}=g_{0}^{n-1}g_{0}l_{1}^{i}=l_{1}^{i}$ and thus $h:\mathbb{N}_{0}\to G$
is a polynomial map of degree $1$ with $h_{0}=1_{G}$. Thus, $H=\langle h\rangle=\langle l_{1}\rangle\cong\mathbb{Z}$
is an infinite cyclic group generated by $l_{1}$. Then, we have $[h]=\{l_{1}^{i}\mid i\in\mathbb{N}_{0}\}\cong\mathbb{N}_{0}$. 

If $g_{0}$ has infinite order, then we claim that there exist $n,m\in\mathbb{Z}\setminus\{0\}$
such that $g_{0}^{n}=l_{1}^{m}$. Indeed, if cosets $g_{0}^{i}\langle l_{1}\rangle$
and $g_{0}^{j}\langle l_{1}\rangle$ had no intersection whenever
$i\ne j$, then this would imply that $\langle l_{1}\rangle$ had
infinite index in $\langle g\rangle$. So for some $i>j$, the $g_{0}^{i}\langle l_{1}\rangle$
and $g_{0}^{j}\langle l_{1}\rangle$ have nonempty intersection and
thus must coincide. Then, $g_{0}^{i}=g_{0}^{j}l_{1}^{m}$ implies
that $g_{0}^{i-j}=l_{1}^{m}$. Then, we can take $n=i-j\ge1$. Again,
we consider $h:=g_{0}^{n-1}g$. Then, $h_{k}=g_{0}^{n-1}g_{0}l_{1}^{k}=l_{1}^{m+k}$
and $\langle h\rangle=\langle l_{1}\rangle\cong\mathbb{Z}$ is a subgroup
of finite index in $\langle g\rangle$. Then, we have $[h]=\{l_{1}^{m+i}\mid i\in\mathbb{N}_{0}\}\cong\mathbb{Z}_{\ge m}$. 

In either case, we consider the group homomorphism $\phi:H\cong\mathbb{Z}\hookrightarrow\mathbb{R}$
and the proper polynomial set $\mathbb{R}_{\ge k}$ for some large
$k$  in $\mathbb{R}$. Then, $\phi^{-1}(\mathbb{R}_{\ge k})$ is
a proper polynomial set in $H$ such that each element in $[h]$ can
be written as a product of at most $1$ element in the sequence $h_{0},h_{1},h_{2},\ldots$
and thus at most $n$ element in the sequence $g_{0},g_{1},g_{2},\ldots$.
\end{proof}

\subsection{The most general case}

Below is a primitive result of the most general case of Waring's problem
for locally nilpotent groups.
\begin{thm}
\label{thm:WP for LNGs} If $g$ has degree $d\ge2$ and $g\mod N$
is non-constant for any normal subgroup $N$ of infinite index in
$\langle g\rangle$, then there exist $A,M\in\mathbb{N}$, a subgroup
$H$ of finite index in $\langle g\rangle$, a polynomial subset $V$
 of $H$ and a polynomial map $p:\mathbb{N}_{0}^{B}\to\mathcal{U}_{n}(\mathbb{Z})$
such that every element in $V\cap p(A\mathbb{N}_{0}^{B})$ can be
written as a product of at most $M$ elements in the sequence $g_{0},g_{1},g_{2},\ldots$.
\end{thm}
We first prove the fundamental case when $\langle g\rangle$ is finitely
generated, torsion-free and nilpotent. 
\begin{proof}
By \cite[Thm 7.5]{HallEdmonton1957}, every finitely generated torsion-free
nilpotent group $\langle g\rangle$ is isomorphic to a subgroup of
$\mathcal{U}_{n}(\mathbb{Z})$ for some $n=n(\langle g\rangle)$.
 Let $B$ be the least upper bound given in the simple observation
and $L'\in\mathbb{N}$ be the least number such that
\[
B\in\{\hat{d}_{i,j}(L)\mid1\le i<j\le n,1\le L\le L'\}.
\]
Then, $B$ cannot be $-\infty$ (resp. $0$), otherwise $g$ has degree
$-\infty$ (resp. $0$), and $B$ cannot be $1$, otherwise Inequalities
(\ref{eq:inequality B}) and \cite[(4.5)]{Hu2020} imply that 
\[
d\le\max\left\{ d_{k_{1},k_{2}}+\cdots+d_{k_{n-1},k_{n}}\mid1=k_{1}\le k_{2}\le\cdots\le k_{n-1}\le k_{n}=n\right\} \le B=1.
\]

Since $B\ge2$,, by Kamke's key theorem, there exist positive integers
$A$, $L''$, and positive numbers $i_{1}$ and $i_{\nu}$, $J_{\nu}$
with $0<i_{\nu}<J_{\nu}$, $\nu=2,3,\ldots,B$, such that for each
$B$ integers $s_{1},\ldots,s_{B}$, divisible by $A$ and subject
to the conditions 
\[
s_{1}\in(i_{1},\infty);\ s_{\nu}/s_{1}^{\nu}\in(i_{\nu},J_{\nu})\text{ for }\nu=2,3,\ldots,B,
\]
the $B$ equations $s_{\nu}=\sum_{\kappa=1}^{L''}x_{\kappa}^{\nu}$,
$\nu=1,2,\ldots,B$, are simultaneously solvable by integers $x_{\kappa}\ge0$.

Let $L=\max\{L',L''\}$ and consider the ordered product 
\[
\hat{g}:=\bigodot_{i=1}^{L}g:\mathbb{N}_{0}^{L}\to\mathcal{U}_{n}(\mathbb{Z});\ (x_{1},\ldots,x_{L})\mapsto g(x_{1})\cdots g(x_{L}).
\]
By \cite[Cor 8]{Hu2020}, $\langle\hat{g}\rangle$ has finite index
in $\langle g\rangle$. Then, by \cite[Thm 14]{Hu2020}, there exists
a finite natural number $M'$ and a sequence $\sigma_{1},\sigma_{2},\ldots,\sigma_{M'}\in S_{L}$,
such that the product 
\[
\tilde{g}=\prod_{i=1}^{M'}\sigma_{i}(\hat{g})=\sigma_{1}(\hat{g})\sigma_{2}(\hat{g})\cdots\sigma_{M'}(\hat{g})
\]
is a symmetric polynomial map in $L$ variables $x_{1},x_{2},\ldots,x_{L}$. 

Since $\langle g\rangle$ is finitely generated and nilpotent, by
a result due to Mal'tsev (cf. \cite[Thm 2.23]{CMZ2017}), $\langle\hat{g}\restriction_{\mathbb{N}_{0}}\rangle=\langle g^{L}\rangle$
has finite index in $\langle g\rangle$ and thus in $\langle\hat{g}\rangle$.
By \cite[Thm 15]{Hu2020} $\langle\tilde{g}\rangle$ has finite index
in $\langle\hat{g}\rangle$ and thus in $\langle g\rangle$.

 Each entry $\tilde{g}_{i,j}$ is a symmetric polynomial of total
degree $\le B$ and can be written as a polynomial expression with
rational coefficients in the power sum symmetric polynomials  $s_{1},s_{2},\ldots,s_{B}$,
i.e., 
\[
\tilde{g}_{i,j}(x_{1},\cdots,x_{N})=p_{i,j}(s_{1},\ldots,s_{B}),
\]
where $p_{i,j}$ is an integer-valued polynomial in $B$ variables
$s_{1},\ldots,s_{B}\in\mathbb{N}_{0}$ with rational coefficients.
Therefore, we can define the following continuous polynomial map in
$B$ variables $s_{1},\ldots,s_{B}\in\mathbb{R}_{\ge0}$: 
\[
p=\begin{pmatrix}1 & p_{1,2} & p_{1,3} & \cdots & p_{1,n}\\
 & 1 & p_{2,3} & \cdots & p_{2,n}\\
 &  & 1 & \ddots & \vdots\\
 &  &  & \ddots & p_{n-1,n}\\
 &  &  &  & 1
\end{pmatrix}:\mathbb{R}_{\ge0}^{B}\to\mathcal{U}_{n}(\mathbb{R}).
\]
Thus, the goal of studying the image of $\tilde{g}:\mathbb{N}_{0}^{N}\to\mathcal{U}_{n}(\mathbb{Z})$
is reduced to studying the image of the restriction of $p:\mathbb{R}_{\ge0}^{B}\to\mathcal{U}_{n}(\mathbb{R})$
on $U(B,N)\cap A\mathbb{N}_{0}^{B}$.  Consider the homomorphism
\[
\phi:H=\langle\tilde{g}\rangle\hookrightarrow\langle g\rangle\hookrightarrow\mathcal{U}_{n}(\mathbb{Z})\hookrightarrow\mathcal{U}_{n}(\mathbb{R}).
\]
Clearly, each element in $\phi^{-1}(p(U(B,N)))\cap\phi^{-1}(p(A\mathbb{N}_{0}^{B}))$
is a product of exactly $M=LM'$ elements in the sequence $g_{0},g_{1},g_{2},\ldots$.

Let $q:\mathbb{R}_{\ge0}^{n}\to\mathbb{R}^{B}$ be the continuous
polynomial map as in \cite[Lem 8]{Hu2020} such that its image has
nonempty interior in $U(B,N)$.  Then, $f:=p\circ q:\mathbb{R}_{\ge0}^{n}\to\mathcal{U}_{n}(\mathbb{R})$
a continuous polynomial map. Hence, $U=f(\mathbb{R}_{\ge0}^{n})$
is a polynomial set in $\mathcal{U}_{n}(\mathbb{R})$ and $V=\phi^{-1}(U)$
is a polynomial set in $\langle g\rangle$ such that each element
in $V\cap\phi^{-1}(p(A\mathbb{N}_{0}^{B}))$ is a product of exactly
$M=LM'$ elements in the sequence $g_{0},g_{1},g_{2},\ldots$. 
\end{proof}
\begin{rem*}
It is not guaranteed that $U=f(\mathbb{R}_{\ge0}^{n})$ is open or
proper in $\mathcal{U}_{n}(\mathbb{R})$; see the degenerate cases
in the proof of Theorem \ref{thm:WP for Heisenberg group}. 
\end{rem*}
Next, we will show how to reduce the most general case to the fundamental
case above. 
\begin{proof}
The goal is to reduce to the case when $G$ is finitely generated
torsion free nilpotent, in particular, when $G=\mathcal{U}_{n}(\mathbb{Z})$.
There are two possible ways to achieve this. If $g_{i}$ has finite
order for all $i$, then $G$ is a finite nilpotent group. So we may
assume that not all $g_{i}$ have finite order. 

An old result of Hirsch implies that every finitely generated nilpotent
group $G$ is isomorphic to a \emph{finite index} subgroup of a direct
product of a finite nilpotent group $T$ and a finitely generated
torsion-free nilpotent group $F$, cf. \cite[Thm 2.22]{Hirsch1938II}
and \cite[Thm 3.21 and Thm 3.23]{Hirsch1946III}.  Let $p_{1}:G\to T\times F\to T$
and $p_{2}:G\to T\times F\to F$ be canonical projections onto direct
summands of $T\times F$. We have two induced polynomial sequences
$t=p_{1}\circ g:\mathbb{N}_{0}\to T$ and $f=p_{2}\circ g:\mathbb{N}\to F$
such that $g=(t,f)$. 

The first idea is to take a periodic subsequence. By \cite[Prop 7]{Hu2020},
$t$ is periodic, say, of period $P\in\mathbb{N}$. Then, we can take
any periodic subsequence $s_{i}:=g_{iP+b}$, for some $0\le b<P$,
such that $p_{1}(s_{i})$ is a constant. Then, $s:\mathbb{N}_{0}\to G$
is a polynomial subsequence of degree not larger than the degree of
$g$, such that $p_{1}(s_{i})$ is a constant of finite order in $T$.
By \cite[Thm 13]{Hu2020}, $\langle s\rangle$ is a \emph{finite index}
subgroup of $G$. Instead of the polynomial sequence $g$, we work
with the periodic polynomial subsequence $s$. So we may assume that
$s:\mathbb{N}_{0}\to\mathbb{Z}/m\times\mathcal{U}_{n}(\mathbb{Z})$
for some $m$ and $p_{1}(s)=\bar{1}\in\mathbb{Z}/m$. Form the following
ordered product
\[
\hat{s}:=\bigodot_{i=1}^{L}s:\mathbb{N}_{0}^{L}\to\mathbb{Z}/m\times\mathcal{U}_{n}(\mathbb{Z});\quad(x_{1},\ldots,x_{L})\mapsto s(x_{1})\cdots s(x_{L}),
\]
and require that $L$ is divisible by $m$. Then, we have $p_{1}(\hat{s})=\bar{0}\in\mathbb{Z}/m$.

The second idea is to replace $g$ by $g^{|T|}=(t^{|T|},f^{|T|})=(1_{T},f^{|T|})$,
where $|T|$ is the cardinal of $T$. Since $G$ is finitely generated
and nilpotent, by a result due to Mal'tsev (cf. \cite[Thm 2.23]{CMZ2017}),
$\langle g^{|T|}\rangle$ is finitely generated torsion free nilpotent
and has finite index in $G$. But notice that one may not be able
to compare the degree of $g^{|T|}$ with the degree of $g$. 

Suppose that there were a normal subgroup $N'$ of infinite index
in $\langle s\rangle$ (resp. $\langle g^{|T|}\rangle$ such that
$s\mod N'$ (resp. $g^{|T|}\mod N'$) is constant. (Notice that $N'$
is not necessarily normal in $G$, as normality is not a transitive
relation.) Since $N'$ is a subgroup of infinite index in $G$, by
\cite[Lem 6]{Hu2020}, $N'$ is contained in a normal subgroup $N$
of infinite index in $G$. This implies that $s\mod N$ (resp. $g^{|T|}\mod N$)
is also constant, which is a contradiction. 

Hence, it suffices to work with $s$ or $g^{|T|}$, which generates
a finitely generated torsion free nilpotent subgroup of finite index
subgroup in $G$. 
\end{proof}
\begin{rem*}
If $\langle g\rangle$ is infinite and has nilpotency class $\le2$,
then the idea given in the previous proof also allows us to reduce
the problem to Theorem \ref{thm:WP for Z^m} or Theorem \ref{thm:WP for Heisenberg group}
and we can show that if $g\mod N$ is non-constant for any normal
subgroup $N$ of infinite index in $\langle g\rangle$, then there
exists $M\in\mathbb{N}$, a subgroup $H$ of finite index in $\langle g\rangle$
and a proper polynomial set $V$ of $H$ such that every element in
$V$ can be written as a product of at most $M$ elements in the sequence
$g_{0},g_{1},g_{2},\ldots$.  
\end{rem*}
\bibliographystyle{alpha}

\end{document}